\documentclass[12pt]{amsart}
\usepackage{amsmath,amsthm,amsfonts,amscd,amssymb,eucal,latexsym,mathrsfs, appendix, yhmath, setspace}
\usepackage{color,colortbl}
\usepackage[numbers,sort&compress]{natbib}
\usepackage[all,cmtip]{xy}
\usepackage{enumerate}

\newtheorem{theorem}{Theorem}[section]
\newtheorem{corollary}[theorem]{Corollary}
\newtheorem{lemma}[theorem]{Lemma}
\newtheorem{proposition}[theorem]{Proposition}

\theoremstyle{definition}
\newtheorem{definition}[theorem]{Definition}
\newtheorem{remark}[theorem]{Remark}

\newtheorem{example}[theorem]{Example}
\newtheorem{question}[theorem]{Question}
\newtheorem{problem}[theorem]{Problem}

 \newcommand{\Map}{{\rm Map}}

\newcommand{\mdim}{{\rm mdim}}
\newcommand{\mdimsm}{\mdim_{\Sigma, \rM}}

\newcommand{\mdims}{{\rm mdim_{\Sigma}}}

\newcommand{\Wdim}{{\rm Wdim}}

\newcommand{\rM}{{\rm M}}
\newcommand{\Sym}{{\rm Sym}}

 \newcommand{\cB}{{\mathcal B}}
  \newcommand{\cC}{{\mathcal C}}
  
  \newcommand{\cF}{{\mathcal F}}
  
  \newcommand{\cK}{{\mathcal K}}

 \newcommand {\cU}{{\mathcal U}}
  \newcommand {\cV}{{\mathcal V}}
    \newcommand {\cW}{{\mathcal W}}
 \newcommand{\cZ}{{\mathcal Z}}

 \newcommand{\bN}{{\mathbb N}}

 \newcommand{\bR}{{\mathbb R}}
 \newcommand{\bT}{{\mathbb T}}
 \newcommand{\bZ}{{\mathbb Z}}

\newcommand{\sU}{{\mathscr U}}
\newcommand{\sV}{{\mathscr V}}

\allowdisplaybreaks

\begin{document}

\title{Conditional sofic mean dimension}

\author[B.~Liang]{Bingbing Liang}
\address[B.~Liang]{School of Mathematical Sciences, Soochow University, Suzhou, Jiangsu 215006, China}
\email{bbliang@suda.edu.cn}


\subjclass[2020]{Primary 37B02, 54E45.}
\keywords{sofic group, conditional mean dimension, $G$-extension, metric mean dimension}

\begin{abstract}
We undertake a study of the conditional mean dimensions for a factor map between continuous actions of a sofic group on two compact metrizable spaces.  When the group is infinitely amenable, all these concepts recover as the conditional mean dimensions introduced in \cite{L22}. A range of results established for actions of amenable groups are extended to the sofic framework.  

Additionally, our exploration encompasses the study of the relative mean dimension introduced by Tsukamoto, shedding light on its inherent correlation with the conditional metric mean dimension within the sofic context. A lower bound on the conditional metric mean dimension, originally proposed by Shi-Tsukamoto, is extended to the sofic case.  
\end{abstract}

\maketitle

\tableofcontents

\section{Introduction}

A classical Hurewicz's inequality states that the topological dimension of a compact metrizable space $X$ can be bounded by the dimension of a continuous image of $X$ and the dimension of corresponding fibers \cite[Theorem VI 7]{HW41}\cite[Theorem 3.3.10]{E95}. There are multiple versions of this inequality in various fields like geometric group theory \cite{BD06} and noncommutative geometry \cite{P22}.

When approaching this inequality from the vantage point of dynamical systems, a pivotal concept is the mean dimension introduced by Gromov \cite{G99m}. The profound implications of mean dimension became evident through the work of Lindenstrauss and Weiss, who not only delved into its intricacies but also constructed a minimal topological dynamical system with mean dimension surpassing one \cite{LW00}. This particular example decisively addressed Auslander's open problem on the embedding of dynamical systems \cite{LW00}. Thus mean dimension emerges as a topological obstruction for the embedding problem of dynamical systems. Since then mean dimension has become a focal point of rigorous investigation not only within the domain of topological dynamics \cite{L99, G15, GLT16, GT20, CDZ22, YCZ22},  but also gathered substantial attention from experts in diverse fields such as geometric analysis \cite{MT15, T18}, operator algebras \cite{LL18, LL19, EZ17}, and information theory \cite{LT18, LT19}.

By a {\it (topological) dynamical system}, denoted by $\Gamma \curvearrowright X$, we mean a continuous action of a discrete group $\Gamma$ on a compact metrizable space $X$.  Let $\pi \colon X \to Y$ be a {\it factor map} between two dynamical systems $\Gamma \curvearrowright X$ and $\Gamma \curvearrowright Y$, i.e. $\pi$ is a surjective continuous $\Gamma$-equivariant map. We may refer to $\Gamma \curvearrowright X$ as an {\it extension system} and $\Gamma \curvearrowright Y$ a {\it factor system}. For a closed subset $ K$ (not necessarily $\Gamma$-invariant) of $X$, as $\Gamma$ is amenable, denotes by $\mdim(K)$  the mean dimension of $K$ under the action of $\Gamma$ (see Definition \ref{closed mdim}).  Motivated by these foundational principles, a dynamical version of Hurewicz's inequality can be posed as follows:

\begin{problem}\label{main interest}
For a factor map $\pi\colon X \to Y$, does the inequality 
$\mdim(X) \leq \mdim(Y)+\sup_{y \in Y} \mdim(\pi^{-1}(y))$
hold true?
\end{problem}

This inquiry originated from Tsukamoto's investigation into the mean dimension of Brody curves \cite[Problem 4.8]{T08}. Recently the problem was negatively solved in \cite[Theorem 1.3]{T22}, where the counterexample was constructed by strengthening Gromov's idea to diminish the dimensions of fibers within $\varepsilon$-resolution. 
In contrast, some affirmative results were made regarding this problem for the so-called $G$-extension systems \cite[Corollary 2.18]{L22}. The notion of $G$-extensions was introduced by Bowen as a generalization of skew product systems \cite{B71}.  More generally, the notion of conditional mean dimension was introduced in \cite{L22} to give a general upper bound of extension systems, and many nice properties of conditional mean dimension were discussed.  
In a parallel exploration, conditional mean dimension takes on a resemblance to topological conditional entropy. Bowen's pivotal work \cite{B71} demonstrated the validity of the topological entropy version of Hurewicz's inequality. 
That is, as $\Gamma=\bZ$, for any factor map $\pi \colon X \to Y$, one has
$$h(X) \leq h(Y) +\sup_{y \in Y} h(\pi^{-1}(y)),$$
where $h(K)$ denotes the topological entropy of a closed subset $K$ of $X$.  
Remarkably, it was established that $\sup_{y \in Y} h(\pi^{-1}(y))$
coincides with the concept of topological conditional entropy, as introduced by Downarowicz and Serafin \cite{DS02, DZ15, Y15}.

The primary objective of this paper is to study the conditional mean dimension  within the broader context of sofic group actions. Our exploration is grounded in the prior examination of conditional mean dimension for actions of amenable groups in \cite{L22}.
The concept of sofic groups, originally introduced by Gromov in his pursuit of Gottschalk's surjectivity conjecture \cite{G99s}, broadens our understanding by encompassing amenable groups and residually finite groups within its class.
By counting the number of good dynamical models compatible with an approximation sequence from a sofic group, Bowen successfully expanded the classical measure entropy theory into the realm of sofic groups \cite{B10}. Using this idea many important dynamical invariants have been extended to the actions of sofic groups, like the topological entropy by Kerr-Li \cite{KL11}, mean dimension by Li \cite{Li13}, topological pressure by Chung \cite{C13}, and conditional entropy by Luo \cite{Lu17}. We may refer to such invariants of sofic group actions collectively as {\it  sofic invariants}.

For actions of amenable groups, it is a good property that conditional mean dimension always bounds the difference of mean dimension between the extension system and factor system \cite[Theoreme 1.2]{L22}. Recall that for the actions of amenable groups, the conditional mean dimension is given by looking at the $\varepsilon$-Width dimensions of all fibers simultaneously (Definition \ref{amenable definition}). Towards the actions of sofic groups, one can also adapt this approach to studying the  $\varepsilon$-Width dimensions restricted to the subset of good dynamical models. However, this pursuit unveils a subtle challenge—the desired bound appears elusive. Instead, by relaxing the condition to study a $\theta$-neighborhood of each fiber, it turns out that one can also give a proper generalization of condition mean dimension into the sofic setting (Theorem \ref{conditional extension}). In fact, this relaxing approach has already appeared in \cite{Lu17} when Luo proved the addition formula for the sofic entropy of $G$-extensions.
More importantly, this relaxed version of conditional mean dimension allows us to establish bounds for the difference of mean dimension between the extension system and factor system (Theorem \ref{general upperbound}).  Since the sofic invariants in general may not decrease when passing to factors, the mean dimension of the factor system in our formula of the bound is replaced with a smaller relative counterpart, which strengthens the estimation of mean dimension.  We mention that this approach has been previously employed in \cite{LL19} when establishing the addition formula of mean dimension for algebraic actions.

Let  $\pi \colon X \to Y$ be a factor map between actions of a sofic group $\Gamma$. 
Our discussion of conditional mean dimension is started with a variant of Gromov's $\varepsilon$-Width dimension, taking into account some continuous pseudmetrics $\rho_X$ and $\rho_Y$ on $X$ and $Y$ respectively. When $\rho_X$ and $\rho_Y$ are both dynamically generating, in Proposition \ref{second reduction}, we show that the conditional mean dimension is independent of the choice of such $\rho_X$ and $\rho_Y$. This independence from the choices of $\rho_X$ and $\rho_Y$ was also emphasized in a comprehensive discussion of conditional topological entropy in \cite{Lu17}. It is worth noting that even in the absolute scenario where $Y$ reduces to a singleton set, this reduction of choosing dynamically generating psedumetrics simplifies the computation of sofic mean dimension significantly. For example, the upper bound of the sofic mean dimension of the Bernoulli full shift $K^\Gamma$  for a general compact metrizable space $K$ can be given in a more straightforward way \cite[Lemma 3.1]{JQ21}.


As a general upper bound on mean (topological) dimension, the metric mean dimension, introduced by Lindenstrauss and Weiss \cite{LW00}, offers a measure of the growth of topological entropy on an $\varepsilon$-scale in proportion to $\log(1/\varepsilon)$. In alignment with the principles underlying conditional topological entropy, we embark on an exploration of the conditional counterpart of metric mean dimension within the context of sofic group actions. Built up with a thorough comparison of two variants of conditional topological entropy as detailed in \cite{Lu17}, in Theorem \ref{equivalent metric definitions}, it concludes that the conditional metric mean dimension can suppress the study of $\theta$-neighborhood of fibers to the genuine fibers.  

In the comparison between the conditional metric mean dimension and its topological equivalent, the concept of relative mean dimension, introduced by Tsukamoto \cite{T22} as a fiberwise adaptation of the conditional mean dimension, is expanded to the sofic context (refer to Definition \ref{relative smdim}). This expansion is shown to be effective in offering a coherent estimation for $G$-extensions in Proposition \ref{relative smdim}. The significance of this fiberwise concept is further highlighted in Theorem \ref{bridge inequality}, where we establish a natural correlation with the conditional metric mean dimension. Consequently, this correlation provides a revised approach to addressing \cite[Question 4.4]{L22}.

Using a sharp waist inequality, Shi and Tsukamoto obtained some lower bounds of conditional metric mean dimension for the integral group actions \cite{ST23, T23}. In particular, it indicates a type of superexponential growth on the fibers of a typical $\Gamma$-equivariant map. Drawing upon Theorem \ref{equivalent metric definitions} and a reduction formula presented in Theorem \ref{suppressing equality}, we extend these estimations to the sofic context in Theorem \ref{waist app2}.

\medskip
\noindent{\it Acknowledgments.}  
I am grateful for the careful reading and detailed comments provided by Professor Hanfeng Li. The author is supported by NSFC grant 12271387.

\section{Preliminaries}
\numberwithin{equation}{section}
\setcounter{equation}{0}

In this section, we prepare some basic terminologies, notations, and list some lemmas for later use. Throughout this paper $\Gamma$ is always a countable discrete group with identity element $e_\Gamma$. Denote by $\cF(\Gamma)$ the collection of nonempty finite subsets of $\Gamma$.
\subsection{Amenable and sofic groups} The group $\Gamma$ is called {\it amenable} if for every $K\in \cF(\Gamma)$ and every $\delta>0$ there exists an $F\in \cF(\Gamma)$ with $|KF\setminus F|<\delta |F|$.  


For a $\bR$-valued function $\varphi$ defined on $\cF(\Gamma)$, we say that   {\it $\varphi(F)$ converges to $c\in \bR$ when $F\in \cF(\Gamma)$ becomes more and more left
invariant}, denoted by 
$$\lim_F\varphi(F)=c,$$
if for any $\varepsilon>0$ there are some $K \in \cF(\Gamma)$ and $\delta > 0$ such that  $|\varphi(F)-c|<\varepsilon$ for all $F \in \cF(\Gamma)$ satisfying $|KF\setminus F|<\delta |F|$. 
More generally, we may consider the following limit supremum over the net of asymptotically invariant finite subsets $F$ of $\Gamma$:
\begin{equation} \label{limsup}
     \limsup_F \varphi(F) :=\inf_{K, \delta > 0} \sup \{ \varphi(F): |KF\setminus F| < \delta, F \in \cF(\Gamma) \}
\end{equation}
 where $K$ and $F$ range over all elements of $\cF(\Gamma)$.

The following strong version of Ornstein-Weiss lemma is due to Gromov \cite[1.3.1]{G99m}\cite[Theorem 4.38]{KLb} \cite[Theorem 9.4.1]{C15}.
\begin{lemma} \label{OW}
Let  $\varphi \colon \cF(\Gamma)  \to [0, +\infty)$ be a function satisfying\\
\begin{enumerate}
    \item $\varphi(Fs)=\varphi(F)$ for all $F \in \cF(\Gamma)$ and $s \in \Gamma$; \\
    \item $\varphi(F_1\cup F_2) \leq \varphi(F_1) + \varphi(F_2)$ for all $F_1, F_2 \in \cF(\Gamma)$.\\
\end{enumerate}
Then the limit $\lim_F\varphi(F)/|F|$ exists.
\end{lemma}

For every $d \in \bN$ write $[d]$ for the set $\{1, 2, \dots, d\}$ and $\Sym(d)$ for the symmetric group over $[d]$. Given a map $\sigma \colon \Gamma \to \Sym(d)$ we adopt the convention $\sigma(s)(v)=\sigma_s(v)$ for every $s \in \Gamma$ and $v \in [d]$.  A sequence of maps $\Sigma=\{\sigma_i: \Gamma \to {\rm \Sym} (d_i)\}_{i \in \bN}$  is called a {\it sofic approximation} for $\Gamma$  if it satisfies:
\begin{enumerate}
\item $\lim_{i\to \infty}|\{v\in [d_i]: \sigma_{i,s}\sigma_{i,t}(v)=\sigma_{i, st}(v)\}|/d_i=1$ for all $s, t\in \Gamma$,

\item $\lim_{i\to \infty}|\{v\in [d_i]: \sigma_{i, s}(v)\neq \sigma_{i,t}(v)\}|/d_i=1$ for all distinct $s, t\in \Gamma$,

\item $\lim_{i\to \infty} d_i=+\infty$.
\end{enumerate}
The group $\Gamma$ is called {\it sofic} if it admits a sofic approximation. 
We shall frequently say that $\sigma$ is a {\it good enough (sofic approximation) } for $\Gamma$ in the sense that for some small $\tau > 0$  and large $F \in \cF(\Gamma)$ (to be specified in the context), the inequality $|\{v \in [d]: \sigma_s\sigma_t(v)=\sigma_{st}(v)\}|/d > 1-\tau$ holds for every $s, t \in F$, and inequality $|\{v \in [d]: \sigma_s(v) \neq \sigma_t(v)\}| > 1-\tau$ holds for every distinct $s, t \in F$.

Every amenable group is sofic since one can use a sequence of asymptotically invariant subsets of the amenable group to construct a sofic approximation. Residually finite groups are also sofic since a sequence of exhausting finite-index subgroups naturally induces a sofic approximation in which case each approximating map is a genuine group homomorphism. We refer the reader to \cite{CCb, CLb} for more information on sofic groups. 

When $\Gamma$ is amenable, we can take advantage of quasitiling property of $\Gamma$ to partition a sufficiently large approximating set into a major part reminiscent of a Rokhlin tower,  together with a small remainder \cite[Lemma 4.6]{KL131}\cite[Lemma 10.36]{KLb}\cite[Lemma 3.2]{Li13}. This paves the way to proving the generalization of sofic invariants.   

\begin{lemma}  \label{amenable generalize}
    Let $\Gamma$ be a countable amenable group.  Fix $0 < \tau, \delta < 1$ and $K \in \cF(\Gamma)$. Then there exist $\ell \in \mathbb{N}$  and $F_1, \dots, F_\ell \in \cF(\Gamma)$ with $|KF_k \setminus F_k| < \delta |F_k|$ for every $k=1, \dots, \ell$, such that for any good enough sofic approximation $\sigma \colon \Gamma \to \Sym(d)$ and any $\mathcal{W} \subseteq [d]$ with $|\mathcal{W}| \geq (1-\tau/2)d$, there exist $\mathcal{C}_1, \dots, \mathcal{C}_\ell \subseteq \cW$ satisfying the following: 
\begin{itemize}
\item[(i)] the map $F_k \times \cC_k \to \sigma(F_k)\cC_k$ sending $(s, c)$ to $\sigma_s(c)$ is bijective for every $k =1, \dots, \ell$;
\item[(ii)] the sets $\{\sigma(F_k)\cC_k\}_{k=1}^\ell$ are pairwise disjoint with $|\bigsqcup_{k=1}^\ell \sigma(F_k)\cC_k| \geq (1-\tau)d$.
\end{itemize}
\end{lemma}

\begin{definition}
    Consider a continuous action $\Gamma \curvearrowright X$ and a continuous pseudometric $\rho_X$ on $X$. Given $F \in \cF(\Gamma), \delta > 0$ and a map $\sigma \colon \Gamma \to \Sym(d)$ for some $d \in \bN$, we set  
$$\Map(\rho_X, F, \delta, \sigma)= \{\varphi \colon [d] \to X:  \rho_{X, 2}(s\varphi, \varphi \circ \sigma(s)) \leq \delta {\rm  \ for \ every \ } s \in F   \},$$
where for every $1 \leq p < \infty$ the pseudometric $\rho_{X, p}$ on $X^{[d]}$  is defined via
$$\rho_{X, p}(\varphi, \psi)=\left(\frac{1}{d}\sum_{v \in [d]} \rho_X(\varphi(v), \psi(v))^p\right)^{1/p}.$$
\end{definition}
Note that $\Map(\rho_X, F, \delta, \sigma)$ is a closed subset of $X^{[d]}$ and its elements can be interpreted as approximately equivariant maps from $[d]$ to $X$. Meanwhile, the pseudometric $\rho_{X, \infty}$ on $X^{[d]}$ is defined via
$$\rho_{X, \infty}(\varphi, \psi)=\max_{v \in [d]} \rho_X(\varphi(v), \psi(v)).$$

The following estimation is used frequently in the sequel \cite[Lemma 2.8]{Li13}.
\begin{lemma} \label{Map lowerbound}
  Let $\rho_X$ be a continuous pseudometric on $X$, $F \in \cF(\Gamma)$ and $\delta >0$. Consider any map $\sigma \colon \Gamma \to \Sym(d)$. Then for each $\varphi \in \Map(\rho_X, F, \delta, \sigma)$, we have
  $$|\{v \in [d]: \rho_X(s\varphi(v), \varphi(\sigma_s(v))) \leq \sqrt{\delta}\} |\geq (1-\delta)d.$$
\end{lemma}

A pseudometric $\rho_X$ on $X$ is called {\it dynamically generating} if 
 $$\sup_{s \in \Gamma} \rho_X(sx, sx') > 0$$
 for any distinct $x, x' \in X$.
 Consider an enumeration of the elements of $\Gamma$ as $s_1=e_\Gamma, s_2, \dots$. For any dynamically generating continuous pseudometric $\rho_X$ on $X$, it induces a compatible metric $\widetilde{\rho_X}$ on $X$ given by
 \begin{equation} \label{compatible induced}
      \widetilde{\rho_X}(x, y)=\sum_{n=1}^\infty \frac{1}{2^n} \rho_X(s_nx, s_ny).
 \end{equation}
 The following lemma is extracted from \cite[Lemma 2.3]{KL13}.
\begin{lemma} \label{Map pseudo}
Let $\rho_X, \rho'_X$ be two continuous pseudometrics on $X$ such that $\rho'_X$ is dynamically generating. Let $F \in \cF(\Gamma)$ and $\delta > 0$. Then there exist $F' \in \cF(\Gamma)$ and $\delta' > 0$ such that 
$$\Map(\rho'_X, F', \delta', \sigma) \subseteq \Map(\rho_X, F, \delta, \sigma)$$
for any good enough map $\sigma \colon \Gamma \to \Sym(d)$.
\end{lemma}

Modifying the argument in the proof of \cite[Lemma 2.10]{Li13}, we have the following:
\begin{lemma} \label{Map into}
  Let $\rho_X, \rho_Y$ be two continuous pseudometrics on $X$ and $Y$ respectively such that $\rho_X$ is dynamically generating. Suppose that $\pi \colon X \to Y$ is a factor map, $F \in \cF(\Gamma)$, and $\delta > 0$. Then there exist $F' \in \cF(\Gamma)$ and $\delta' > 0$ such that   for any good enough map $\sigma \colon \Gamma \to \Sym(d)$, one has
  $$\pi^d(\Map(\rho_X, F', \delta', \sigma)) \subseteq \Map(\rho_Y, F, \delta, \sigma),$$ 
where $\pi^d \colon X^d \to Y^d$ sends $\varphi$ to $\pi \circ \varphi$.
\end{lemma}

\begin{proof}
    Let $\widetilde{\rho_X}$ be the compatible metric induced from $\rho_X$ as in (\ref{compatible induced}). According to \cite[Lemma 2.10]{Li13}, there exists $\delta_1 > 0$ such that 
    $$\pi^d(\Map(\widetilde{\rho_X}, F, \delta_1, \sigma)) \subseteq \Map(\rho_Y, F, \delta, \sigma)$$
    for any map $\sigma \colon \Gamma \to \Sym(d)$.
Applying Lemma \ref{Map pseudo} to such $F$ and $\delta_1$, we find $F' \in \cF(\Gamma)$ and $\delta' > 0$ such that $\Map(\rho_X, F', \delta', \sigma) \subseteq  \Map(\widetilde{\rho_X}, F, \delta_1, \sigma)$ as $\sigma$ is good enough. The desired containment then follows.
\end{proof}

Throughout the rest of this paper, $\Gamma$ will consistently denote a countable sofic group associated with a sofic approximation $\Sigma$.

\subsection{$G$-extensions}

\begin{definition}{\cite[Page 410]{B71}} \label{G-extension}
Let $\pi\colon X \to Y$ be a factor map and $\Gamma \curvearrowright G$ be another dynamical system. We say $X$ is a {\it $G$-extension of $Y$}, and $\pi$ is a {\it $G$-extension map}, if there exists a continuous map $ X \times G \to X$ sending $(x, g)$ to $xg$ such that for each $x\in X, g, g' \in G$ and $s \in \Gamma$, the following holds:
\begin{enumerate}
    \item $\pi^{-1}(\pi(x))=xG$;
    \item $xg=xg'$ exactly when $g=g'$;
    \item $s(xg)=(sx)(sg)$.
\end{enumerate}
\end{definition}
Note that when $G$ is a group and the action  $\Gamma \curvearrowright G$ is trivial, the factor map $\pi$ recovers as a so-called {\it principal group extension}.

\begin{example}  \label{G-example}
\begin{enumerate}
    \item Projection maps inducing simple $G$-extensions: Let $\Gamma \curvearrowright Y$ and $\Gamma \curvearrowright Z$ be two dynamical systems. The projection map from the product space $X:=Y\times Z$ carrying the diagonal action induces a simple $G$-extension via the map $ X \times G \to X$ sending $((y, z), g)$ to $(y, g)$ for every $(y,z) \in X=Y\times Z$ and $g \in G=Z$;\\
    \item $G$-extensions through cocycles functions: Let $\Gamma \curvearrowright Y$ and $\Gamma \curvearrowright G$ be two dynamical systems such that $G$ is a compact group and $\Gamma$ acts on $G$ by continuous automorphisms. A (continuous) {\it cocycle} for $\Gamma \curvearrowright Y$ and $\Gamma \curvearrowright G$ is  a continuous map $\sigma\colon \Gamma \times Y \to G$ defined by the relation  
$$\sigma(st, y)=\sigma(s, ty)\cdot s(\sigma(t, y))$$
for every $s, t \in \Gamma$ and $y \in Y$.
This map induces an action of $\Gamma$ on $Y \times G$ by
$$s(y, g):=(sy, \sigma(s,y)\cdot (sg))$$
for all $s \in \Gamma, y \in Y$ and $g\in G$.
Consequently, $Y\times G$ becomes a $G$-extension of $Y$ via the map $ (Y \times G) \times G \to Y\times G$ sending $((y, g), h)$ to $(y, gh)$. We denote this $G$-extension derived from such a cocycle $\sigma$ as $Y\times_\sigma G$;\\
   \item $G$-extensions from algebraic actions: Another notable class of $G$-extension arises from algebraic actions. 
Recall that a dynamical system $\Gamma \curvearrowright X$ is called an {\it algebraic action} if $X$ is a compact metrizable group and $\Gamma$ acts on $X$ by continuous automorphisms. Let $\pi\colon X \to Y$ be a factor map between algebraic actions such that $\pi$ is a group homomorphism. Put $G=\ker (\pi)$. In this context, $X$ stands as a $G$-extension given via the map sending $(x, g) \in X \times G$ to $xg$ as the multiplication of group elements.
\end{enumerate}
\end{example}

\section{A reformulation of conditional mean dimension}
\numberwithin{equation}{section}
\setcounter{equation}{0}

In this section, we introduce a variant of the conditional mean dimension for actions of amenable groups. Subsequently, we establish its equivalence with the conventional conditional mean dimension.  This equivalence will lay the groundwork for the subsequent exploration of conditional mean dimension in the context of sofic groups in Section 4. Throughout this section $\Gamma$ denotes a countable amenable group.

Let $X, P$ denote two compact metrizable spaces and $\rho_X$ a continuous pseudometric on $X$. Fix $\varepsilon > 0$. Recall that a continuous map $f \colon X \to P$ is an {\it $(\varepsilon, \rho_X)$-embedding} if for every $x, x' \in X$ with $f(x)=f(x')$, one has $\rho_X(x, x') < \varepsilon$. Building upon this notion, we introduce two relative variants of this notion.

\begin{definition}
Let  $\pi \colon X \to Y$ be a  factor map between two dynamical systems. Consider $\rho_Y$ as a continuous pseudometric on $Y$ and fix $\theta > 0$.   We say a continuous map $f \colon X \to P$ is an {\it $(\varepsilon, \theta, \rho_X|\rho_Y)$-embedding (with respect to $\pi$)} if 
$$\rho_X(x, x') < \varepsilon$$
for every $x, x' \in X$ satisfying that $f(x)=f(x')$ and $\rho_Y(\pi(x), \pi(x')) \leq \theta$.
We call  $f$ is an {\it $(\varepsilon, \rho_X|Y)$-embedding} if $\rho_X(x, x') < \varepsilon$ holds for every $x, x' \in X$ satisfying that $f(x)=f(x')$ and $\pi(x)=\pi(x')$. 

Denote by $\Wdim_\varepsilon(\rho_X|\rho_Y, \theta)$ the minimal (covering) dimension $\dim (P)$ of a compact metrizable space $P$ that admits an $(\varepsilon, \theta, \rho_X|\rho_Y)$-embedding from $X$ to $P$. Similarly,, denote by $\Wdim_\varepsilon(\rho_X|Y)$ the minimal (covering) dimension $\dim (Q)$ of a compact metrizable space $Q$ that admits an $(\varepsilon, \rho_X|Y)$-embedding from $X$ to $Q$. When Y is a singleton, we simplify $\Wdim_\varepsilon(\rho_X|Y)$ to $\Wdim_\varepsilon(\rho_X)$.
\end{definition}

For each $F \in \cF(\Gamma)$ it induces a continuous pseudometric $\rho_{X, F}$ on $X$ via
$$\rho_{X, F}(x, x'):=\max_{s \in F} \rho_X(sx, sx').$$
Consider the function sending $F \in \cF(\Gamma)$ to $\varphi_\theta(F):=\Wdim_\varepsilon(\rho_{X, F}|\rho_{Y, F}, \theta)$. It can be readily verified that $\varphi_\theta$ satisfies the conditions of Lemma \ref{OW}, leading to the existence of the limit 
$$\mdim_\varepsilon(\rho_X|\rho_Y, \theta) :=\lim_F\frac{\Wdim_\varepsilon(\rho_{X, F}|\rho_{Y, F}, \theta)}{|F|}$$
exists\footnote{note that $\varphi_\theta$ may fail the monotonicity!}.
Observe that $\mdim_\varepsilon(\rho_X|\rho_Y, \theta)$ is monotone in variables $\varepsilon$ and $\theta$ respectively. Set
$$\mdim_\varepsilon(\rho_X|\rho_Y):=\inf_{\theta > 0} \mdim_\varepsilon(\rho_X|\rho_Y, \theta).$$

Meanwhile, the function $\cF(\Gamma) \to [0, \infty)$ sending $F$ to $\Wdim_\varepsilon(\rho_{X, F}|Y)$ satisfies the conditions of Lemma \ref{OW} as well and we set 
$$\mdim_\varepsilon(\rho_X|Y) :=\lim_F\frac{\Wdim_\varepsilon(\rho_{X, F}|Y)}{|F|}$$
When Y is a singleton, we simplify $\mdim_\varepsilon(\rho_X|Y)$ to $\mdim_\varepsilon(\rho_X)$.

\begin{definition} \label{amenable definition}
    For any continuous pseudometrics $\rho_X$ and $\rho_Y$ on $X$ and $Y$ respectively,  we define 
    $$\mdim(\rho_X|\rho_Y):=\sup_{\varepsilon > 0} \mdim_\varepsilon(\rho_X|\rho_Y)$$
    and 
      $$\mdim(\rho_X|Y):=\sup_{\varepsilon > 0} \mdim_\varepsilon(\rho_X|Y).$$

 When $\rho_X$ and $\rho_Y$ are compatible metrics the {\it conditional mean dimension $\mdim(X|Y)$ of $\Gamma \curvearrowright X$ relative to $\Gamma \curvearrowright Y$}  is defined as the value $\mdim(\rho_X|Y)$ (see \cite[Definition 2.3, Proposition 2.5]{L22}).
By compactness of $X$, we see that $\mdim(\rho_X|Y)$ is independent of the choice of compatible metric $\rho_X$.

 Meanwhile, by compactness of $X$ and $Y$, the variant $\mdim(\rho_X|\rho_Y)$ is independent of the choice of compatible metrics $\rho_X$ and $\rho_Y$.  
When $Y$ reduces to a singleton, $\mdim(X|Y)$ recovers as the {\it  mean dimension $\mdim(X)$ of $\Gamma \curvearrowright X$} (see \cite[Section 1.5]{G99m}). To emphasize the metric we may write $\mdim(\rho_X)$ for $\mdim(X)$.
\end{definition}

For a closed subset $K$ (not necessarily $\Gamma$-invariant) of $X$, in light of  (\ref{limsup}), we can define its mean dimension by taking limit supremum over the net of asymptotically invariant finite subsets $F$ of $\Gamma$.

\begin{definition} \label{closed mdim}
   Fix a compatible metric $\rho_X$ on $X$. The {\it mean dimension of a closed subset $K$} of $X$ is defined as
    $$\mdim(K):=\sup_{\varepsilon > 0} \limsup_F \frac{\Wdim_\varepsilon(K, \rho_{X, F})}{|F|}.$$
\end{definition}

Taking advantage of quasi-tiling property of amenable groups, we establish the following equivalence, adapting the argument presented in  \cite[Theorem 1.2]{L22}. 

\begin{proposition}\label{alternative}
Let $\Gamma$ be an amenable group, and consider $\rho_X$ and $\rho_Y$ compatible metrics on $X$ and $y$ respectively. Then for any factor map $\pi \colon X \to Y$, we have $\mdim(\rho_X|Y)=\mdim(\rho_X|\rho_Y).$
\end{proposition}

\begin{proof}
The direction that $\mdim(\rho_X|Y) \leq \mdim(\rho_X|\rho_Y)$ is clear by definition. We shall show the converse. For every $\varepsilon > 0$ it suffices to prove 
$$\mdim_\varepsilon(\rho_X|\rho_Y) \leq \mdim_\varepsilon(\rho_X|Y).$$

Denote by $\cB(K, \delta)$ the set of all $F \in \cF(\Gamma)$ satisfying $|\{t \in F: Kt \subseteq F\}| \geq (1-\delta)|F|$. Let $\tau > 0$. Choose $K \in \cF(\Gamma)$ and $\tau > \varepsilon_0 > 0$ good enough such that whenever $F \in \cB(K, \varepsilon_0)$ one has 
\begin{equation} \label{sufficient approximate}
\Wdim_\varepsilon(\rho_{X, F}|Y) \leq (\mdim_\varepsilon(\rho_X|Y)+\tau)|F|.
\end{equation}
Applying the quasi-tiling lemma \cite[Lemma 4.4]{LT14} to such $\varepsilon_0$ and $K$, there exist $\delta > 0, K' \in \cB(\Gamma)$, and $F_1, \dots, F_m \in \cB(K, \varepsilon_0)$ satisfying the property: for every $A \in \cB(K', \delta)$, there exist $D_1, \dots, D_m \in \cF(\Gamma)$ such that $\{F_jc_j\}_{1\leq j \leq m, c_j \in D_j}$ are pairwise disjoint subsets of $A$ and 
\begin{equation} \label{small reminder}
    |A\setminus \cup_{j=1}^m F_jD_j| \leq \varepsilon_0|A|.
\end{equation}

For each $j=1, \dots, m$ choose $f_j \colon X \to P_j$ as an $(\varepsilon|Y, \rho_{X, F_j})$-embedding with $\dim P_j=\Wdim_\varepsilon(\rho_{X, F_j}|Y)$. By compactness there exists $\theta > 0$ depending only on $\{F_j\}_{j=1}^m$ such that for each $j=1, \dots, m$ and $x, x' \in X$, whenever $\rho_{Y, F_j}(\pi(x), \pi(x')) \leq \theta$ and $f_j(x)=f_j(x')$, one has 
$\rho_{X, F_j}(x, x') < \varepsilon$. That is, each $f_j$ is an $(\varepsilon, \theta, \rho_{X, F_j}|\rho_{Y, F_j})$-embedding. 

Now consider the continuous map $\Phi \colon X \to \Pi_{j=1}^m P_j^{D_j}$ sending $x$ to $((f_j(rx))_{r \in D_j})_{1\leq j \leq m}$.
Denote by $A^\circ$ the set $\cup_{j=1}^m F_jD_j$.
By our choice of $\theta$ it is straightforward to verify that  $\Phi$ is an $(\varepsilon, \theta, \rho_{X, A^\circ}|\rho_{Y, A^\circ})$-embedding. Then, by \eqref{sufficient approximate}, we have
\begin{align*}
    \Wdim_\varepsilon(\rho_{X, A^\circ}|\rho_{Y, A^\circ}, \theta) & \leq \dim (\Pi_{j=1}^m P_j^{D_j})\\
    & \leq \sum_{j=1}^m |D_j||F_j|\frac{\Wdim_\varepsilon(\rho_{X, F_j}|Y)}{|F_j|} \\
    &\leq |A|(\mdim_\varepsilon(\rho_X|Y)+\tau). 
\end{align*}
Denote by $B$ the subtraction $A \setminus A^\circ$.
From \eqref{small reminder}, we obtain
\begin{align*}
    \Wdim_\varepsilon(\rho_{X, A}|\rho_{Y, A}, \theta) & \leq \Wdim_\varepsilon(\rho_{X, A^\circ}|\rho_{Y, A^\circ}, \theta) + \Wdim_\varepsilon(\rho_{X, B}|\rho_{Y, B}, \theta) \\
    & \leq |A|(\mdim_\varepsilon(\rho_X|Y)+\tau) +|B| \Wdim_\varepsilon(\rho_X|\rho_Y, \theta) \\
    & \leq |A|(\mdim_\varepsilon(\rho_X|Y)+\tau) +\tau|A|\Wdim_\varepsilon(\rho_X)
\end{align*}
Since $A \in \cB(K', \delta)$ is arbitrary, we have
$$\mdim_\varepsilon(\rho_X|\rho_Y) \leq \mdim_\varepsilon(\rho_X|\rho_Y, \theta) \leq \mdim_\varepsilon(\rho_X|Y)+\tau+ \tau\Wdim_\varepsilon(\rho_X).$$
Letting $\tau$ tend to zero, it follows that $\mdim_\varepsilon(\rho_X|\rho_Y) \leq \mdim_\varepsilon(\rho_X|Y).$

\end{proof}

Taking advantage of Proposition \ref{alternative}, we can present a concise proof of \cite[Theorem 1.2]{L22}.
\begin{corollary}
For any factor map $\pi \colon X \to Y$, we have
$$\mdim(X) \leq \mdim(Y) +\mdim(X|Y).$$
\end{corollary}

\begin{proof}
Fix  some compatible metrics $\rho_X, \rho_Y$ on $X$ and $Y$ respectively. In light of Proposition \ref{alternative}, it is equivalent to show for every $\varepsilon > 0$ one has
$$\mdim_\varepsilon(\rho_X) \leq \mdim_\varepsilon(\rho_X|\rho_Y) + \mdim(\rho_Y).$$

  Fix  $\varepsilon, \theta > 0$ and $F  \in \cF(\Gamma)$. Suppose that $f \colon X \to P$ is an $(\varepsilon, \theta, \rho_{X, F}|\rho_{Y, F})$-embedding, and $g \colon Y \to Q$ is an $(\varepsilon, \rho_{Y, F})$-embedding.  Consider the continuous map $\varphi \colon X \to P\times Q$ sending $x$ to $(f(x), g\circ \pi(x))$. For every $x, x' 
\in X$ with $\varphi(x)=\varphi(x')$, since $g$ is a $(\theta, \rho_{Y, F})$-embedding, we obtain $\rho_{Y, F}(\pi(x), \pi(x')) < \theta$. Since $f$ is an $(\varepsilon, \theta, \rho_{X, F}|\rho_{Y, F})$-embedding, we morerover obtain that $\rho_{X, F}(x, x') < \varepsilon$. This concludes that $\varphi$ is an $(\varepsilon, \rho_{X, F})$-embedding and hence
$$\Wdim_\varepsilon(\rho_{X, F}) \leq \Wdim_\varepsilon(\rho_{X, F}|\rho_{Y, F}, \theta) + \Wdim_\theta(\rho_{Y, F}).$$
   
   Now as $F$ becomes more and more left-invariant, we get
$$\mdim_\varepsilon(\rho_X) \leq \mdim_\varepsilon(\rho_X|\rho_Y, \theta) +\mdim_\theta(\rho_Y).$$
Letting $\theta$ approach zero, the desired inequality follows. 

\end{proof}

Recently Tsukamoto constructed a factor map of $\bZ$-actions indicating that a dynamical version of Hurewicz's inequality fails \cite[Theorem 1.3]{T22}. In light of this example and our positive estimation for the mean dimension of the extension system, it concludes that the difference between $\mdim(X|Y)$ and $\sup_{y \in Y} \mdim(\pi^{-1}(y))$ can be arbitrarily large.

To end this section, we discuss some reduction properties of conditional mean dimension. It turns out that the computation of $\mdim(X|Y)$ can be streamlined by considering $\mdim(\rho_X|\rho_Y)$, provided that both $\rho_X$ and $\rho_Y$ are dynamically generating.   

\begin{proposition} 
Let  $\rho_X$ and $\rho_X'$ be two continuous pseudometrics on $X$, and $\rho_Y$ a continuous pseudometrics on $Y$. If both $\rho_X$ and $\rho_X'$ are dynamically generating, then for any factor map $\pi \colon X \to Y$, we have
$$\mdim(\rho_X|\rho_Y)=\mdim(\rho_X'|\rho_Y).$$
In particular, we have $\mdim(X|Y)=\mdim(\rho_X|\rho_Y)$ if $\rho_Y$ is a compatible metric.
\end{proposition}

\begin{proof}
    By symmetry it suffices to show $\mdim_\varepsilon(\rho_X'|\rho_Y) \leq \mdim(\rho_X|\rho_Y)$ for every $\varepsilon > 0$.
Given that $\rho_X$ is dynamically generating, there exist $K \in \cF(\Gamma)$ containing $e_\Gamma$ and $\varepsilon' > 0$ such that whenever $\rho_{X, K}(x, x') <\varepsilon'$ one has $\rho_X'(x, x') < \varepsilon$.
Let $f \colon X \to P$ be an $(\varepsilon', \theta, \rho_{X, KF}|\rho_{Y, KF})$-embedding and $g \colon Y \to Q$  a $(\theta, \rho_{Y, KF\setminus F})$-embedding. Define a continuous map $\Phi \colon X \to P\times Q$ sending $x$ to $(f(x), g\circ \pi(x))$. From the choices of $f$ and $g$ one can readily deduce that $\Phi$ is an $(\varepsilon, \theta, \rho_{X, F}'|\rho_{Y, F})$-embedding. Therefore,
\begin{align*}
   &\Wdim_\varepsilon(\rho_{X, F}'|\rho_{Y, F}, \theta) \\
  & \leq \Wdim_{\varepsilon'}(\rho_{X, KF}|\rho_{Y, KF}, \theta) + \Wdim_\theta(\rho_{Y, KF\setminus F})\\
  & \leq \Wdim_{\varepsilon'}(\rho_{X, F}|\rho_{Y, F}, \theta) + |KF\setminus F|(\Wdim_{\varepsilon'}(\rho_X|\rho_Y, \theta) + \Wdim_\theta(\rho_Y)).
\end{align*}

Dividing both sides by $|F|$ and let $F$ get more and more left-invariant, it concludes that
$$\mdim_\varepsilon(\rho'_X|\rho_Y, \theta) \leq \mdim_{\varepsilon'}(\rho_X|\rho_Y, \theta).$$
Letting $\theta$ tend to zero, we obtain
$$\mdim_\varepsilon(\rho_X'|\rho_Y) \leq \mdim_{\varepsilon'}(\rho_X|\rho_Y) \leq \mdim(\rho_X|\rho_Y).$$
 
\end{proof}

\begin{proposition} \label{amenable reduction}
Fix a continuous pseudometric $\rho_X$ on $X$. Let $\rho_Y$ and $\rho_Y'$ be two continuous pseudometrics on $X$  which are dynamically generating. Then, for any factor map $\pi \colon X \to Y$, we have
$$\mdim(\rho_X|\rho_Y)=\mdim(\rho_X|\rho_Y').$$
In particular, we have $\mdim(X|Y)=\mdim(\rho_X|\rho_Y)$ if $\rho_X$ is dynamically generating.
\end{proposition}

\begin{proof}
    Fix $\varepsilon, \theta > 0$. By symmetry it suffices to show $$\mdim_\varepsilon(\rho_X|\rho_Y) \leq \mdim_\varepsilon(\rho_X|\rho'_Y, \theta).$$
Since $\rho_Y$ is dynamically generating, there exist $K \in \cF(\Gamma)$ and $\theta' > 0$ such that whenever $\rho_{Y, K}(y, y') \leq \theta'$ one obtain $\rho_Y'(y, y') \leq \theta$.

Let $f \colon X \to P$ be an $(\varepsilon, \theta, \rho_{X, F}|\rho'_{Y, F})$-embedding and $g \colon Y \to Q$  a $(\theta', \rho_{Y, KF\setminus F}$-embedding. Then it is readily checked that the map $X \to P\times Q$ sending $x$ to $(f(x), g\circ \pi(x))$ is an $(\varepsilon, \theta', \rho_{X, F}|\rho_{Y, F})$-embedding. Therefore, we have
\begin{align*}
    \Wdim_\varepsilon(\rho_{X, F}|\rho_{Y, F}, \theta') 
    & \leq \Wdim_\varepsilon(\rho_{X, F}|\rho'_{Y, F}, \theta) + \Wdim_{\theta'}(\rho_{Y, KF\setminus F})\\
    & \leq \Wdim_\varepsilon(\rho_{X, F}|\rho'_{Y, F}, \theta) + | KF\setminus F|\Wdim_{\theta'}(\rho_{Y}).\\
\end{align*}
Dividing both sides by $|F|$ and let $F$ get more and more left-invariant, we conclude
$$\mdim_\varepsilon(\rho_X|\rho_Y) \leq \mdim_\varepsilon(\rho_X|\rho_Y, \theta') \leq \mdim_\varepsilon(\rho_X|\rho'_Y, \theta).$$
    
\end{proof}

\section{Conditional sofic mean dimension}
\numberwithin{equation}{section}
\setcounter{equation}{0}

In this section, we introduce the concepts of conditional mean dimension and its fiberwise variant for actions of a sofic group and establish their fundamental properties. Notably,  we demonstrate two reduction formulas and establish that these concepts extends their counterparts within the context of amenable group actions. Throughout the rest of this paper, we consider a fixed factor map $\pi \colon X \to Y$ between two continuous actions of a sofic group $\Gamma$, equipped with continuous pseudometrics $\rho_X$ and $\rho_Y$ on $X$ and $Y$ respectively.

\begin{definition}
For fixed $\varepsilon, \theta > 0, F \in \cF(\Gamma)$, consider a map $\sigma \colon \Gamma \to \Sym(d)$ and a subset $\mathcal{K} \subseteq X^{[d]}$. We say a continuous map $f \colon \mathcal{K} \to P$ for a compact metrizable space $P$ is an {\it $(\varepsilon, \theta, \rho_{X, \infty}|\rho_{Y, \infty})$-embedding} if, for every pair $\varphi, \varphi' \in \mathcal{K}$ satisfying  $f(\varphi)=f(\varphi')$ and $\rho_{Y, \infty}(\pi\circ \varphi, \pi\circ \varphi') \leq \theta$, one has 
$$\rho_{X, \infty}(\varphi, \varphi') < \varepsilon.$$
We call $f$ is an {\it $(\varepsilon|Y, \rho_{X, \infty})$-embedding} if $\rho_{X, \infty}(\varphi, \varphi') < \varepsilon$ holds for every  $\varphi, \varphi' \in \mathcal{K}$ satisfying  $f(\varphi)=f(\varphi')$ and $\pi\circ \varphi =\pi\circ \varphi'$. 
\end{definition}

Denote by $\Wdim_\varepsilon(\mathcal{K}, \theta, \rho_{X, \infty}|\rho_{Y, \infty})$ the minimal (covering) dimension $\dim P$ of a compact metrizable space  $P$ that admits an $(\varepsilon, \theta, \rho_{X, \infty}|\rho_{Y, \infty})$-embedding from $\mathcal{K}$ to $P$. Similarly, denote by $\Wdim_\varepsilon(\mathcal{K}, \rho_{X, \infty}|Y)$ the minimal (covering) dimension $\dim Q$ of a compact metrizable space  $Q$ that admits an $(\varepsilon, \theta, \rho_{X, \infty}|Y)$-embedding from $\mathcal{K}$ to $Q$. If $Y$ is a singleton, we  simplify $\Wdim_\varepsilon(\mathcal{K}, \rho_{X, \infty}|Y)$ as $\Wdim_\varepsilon(\mathcal{K}, \rho_{X, \infty})$.

Let $\Sigma=\{\sigma_i \colon \Gamma \to \Sym(d_i)\}_{i=1}^\infty$ be a sofic approximation of $\Gamma$.
\begin{definition} \label{top def of mdim}
 For every $\varepsilon, \theta > 0, F \in \cF(\Gamma)$, we define
$$\mdim_\Sigma^\varepsilon(\rho_X|\rho_Y, \theta, F, \delta):=\limsup_{i\to \infty}\frac{1}{d_i}\Wdim_\varepsilon(\Map(\rho_X, F, \delta, \sigma_i), \theta, \rho_{X, \infty}|\rho_{Y, \infty}).$$
If  $\Map(\rho, F, \delta, \sigma_i)$ is empty for all sufficiently large $i$, we set 
$\mdim_\Sigma^\varepsilon(\rho_X|\rho_Y, \theta, F, \delta)=-\infty.$
Set
$$\mdim_\Sigma^\varepsilon(\rho_X|\rho_Y, \theta):=\inf_{F \in \cF(\Gamma)} \inf_{\delta > 0}  \mdim_\Sigma^\varepsilon(\rho_X|\rho_Y, \theta, F, \delta),$$
$$\mdim_\Sigma^\varepsilon(\rho_X|\rho_Y):=\inf_{\theta > 0} \mdim_\Sigma^\varepsilon(\rho_X|\rho_Y, \theta),$$
and 
$$\mdim_\Sigma(\rho_X|\rho_Y):=\sup_{\varepsilon > 0} \mdim_\Sigma^\varepsilon(\rho_X|\rho_Y).$$

 When $\rho_X$ and $\rho_Y$ are compatible metrics, we define the {\it conditional sofic mean dimension $\mdims(X|Y)$ of $\Gamma \curvearrowright X$ relative to $\Gamma \curvearrowright Y$}  as  the value $\mdim_\Sigma(\rho_X|\rho_Y)$.  By Lemma \ref{Map pseudo} and an argument of compactness, one sees that $\mdims(X|Y)$ is independent of the choice of compatible metrics $\rho_X$ and $\rho_Y$.
 When $Y$ is trivial as a singleton, $\mdims(X|Y)$ recovers as the {\it sofic mean dimension $\mdims(X)$ of $\Gamma \curvearrowright X$} (see \cite[Definition 2.4]{Li13} \cite[Definition 2.12]{H17}).  To emphasize the metric we may write  $\mdims(\rho_X)$ for $\mdims(X)$.

 Meanwhile,  we set
 $$\mdims(\rho_X|Y):= \sup_{\varepsilon > 0} \inf_{F \in \cF(\Gamma)} \inf_{\delta > 0} \limsup_{i\to \infty}\frac{1}{d_i}\Wdim_\varepsilon(\Map(\rho_X, F, \delta, \sigma_i), \rho_{X, \infty}|Y).$$
  When $\varepsilon, F$ and  $\delta$ are fixed, we write $\mdim_\Sigma^\varepsilon(\rho_X|Y, F, \delta)$ for the above limit supremum. If  $\Map(\rho, F, \delta, \sigma_i)$ is empty for all sufficiently large $i$, the limit supremum is set to be $-\infty$. In a similar situation as above, we have the notion $\mdim_\Sigma^\varepsilon(\rho_X|\rho_Y)$. 
\end{definition}

\begin{remark}
    One can also express the conditional sofic mean dimension using finite open covers in the spirit of \cite[Definition 2.3]{L22}.
In Proposition \ref{second reduction} we shall prove that $\mdims(X|Y)$ can be computed in terms of $\mdims(\rho_X|\rho_Y)$ for any dynamically generating continuous pseudometrics $\rho_X$ and $\rho_Y$ on $X$ and $Y$ respectively.\end{remark}

Meanwhile, we have the sofic mean dimension $\mdim_\Sigma(Y|X)$ of $\Gamma \curvearrowright Y$ relative to $\Gamma \curvearrowright X$  which appeared first in \cite[Definition 8.6]{LL19}.
\begin{definition}
For every $\varepsilon > 0, F \in \cF(\Gamma)$, we set
$$\mdim_\Sigma^\varepsilon(\rho_Y|\rho_X, F, \delta):=\limsup_{i\to \infty}\frac{1}{d_i}\Wdim_\varepsilon(\pi^{d_i}(\Map(\rho_X, F, \delta, \sigma_i)),\rho_{Y, \infty}).$$
If  $\Map(\rho, F, \delta, \sigma_i)$ is empty for all sufficiently large $i$, we set 
$\mdim_\Sigma^\varepsilon(\rho_Y|\rho_X, F, \delta)=-\infty.$
Set
$$\mdim_\Sigma^\varepsilon(\rho_Y|\rho_X):=\inf_{F \in \cF(\Gamma)} \inf_{\delta > 0} \mdim_\Sigma^\varepsilon(\rho_Y|\rho_X, F, \delta),$$
and
$$\mdim_\Sigma(\rho_Y|\rho_X):=\sup_{\varepsilon > 0} \mdim_\Sigma^\varepsilon(\rho_Y|\rho_X).$$

 When $\rho_X$ and $\rho_Y$ are compatible metrics the  {\it sofic mean dimension $\mdim_\Sigma(Y|X)$ of $\Gamma \curvearrowright Y$ relative to $\Gamma \curvearrowright X$} is defined as the value $\mdim_\Sigma(\rho_Y|\rho_X)$. By Lemma \ref{Map pseudo} and an argument of compactness,  $\mdim_\Sigma(\rho_Y|\rho_X)$ is independent of the choice of compatible metrics $\rho_X$ and $\rho_Y$. From the definition, clearly we have
 $$\mdims(\rho_X|Y) \leq \mdims(\rho_X|\rho_Y).$$

Using the technique in the proof of \cite[Lemma 3.7]{Li13},  we now prove two reduction formulas for the conditional sofic mean dimension as appeared in the amenable case.
\end{definition}

\begin{proposition}  \label{dynamical reduction}
 If $\rho_X$ and $\rho_X'$ are two continuous pseudometrics on $X$ which are dynamically generating and $\rho_Y$ is a continuous pseudometric on $Y$, then for any factor map $\pi \colon X \to Y$, we have
$$\mdims(\rho_X|\rho_Y)=\mdims(\rho_X'|\rho_Y).$$
In particular, we have $\mdims(X|Y)=\mdims(\rho_X|\rho_Y)$ if $\rho_Y$ is a compatible metric.
\end{proposition}

\begin{proof}
Fix $\beta >0$. We will show $\mdim_\Sigma^\varepsilon(\rho'_X|\rho_Y) \leq \mdims(\rho_X|\rho_Y)+ 2\beta$ for every $\varepsilon > 0$. By symmetry, the conclusion then follows.

Since $\rho_X$ is dynamically generating, there exist  $K \in  \cF(\Gamma)$ and $\varepsilon' > 0$ such that for every $x, x' \in X$ satisfying that  $ \rho_{X, K}(x, x') < \varepsilon'$, then one has  $\rho'_X(x, x') < \varepsilon$.
Let $\sU$ be a finite open cover of $X$ with ${\rm diam}(U, \rho'_X) < \varepsilon$ for every $U \in \sU$. Take  $F \in  \cF(\Gamma)$ containing $K$, $\theta > 0$, and $\delta > 0$ with $\sqrt{\delta} < \varepsilon'/3$ and $\delta |\sU||K| < \beta$ such that 
\begin{align}\label{Fdelta}
\mdim_\Sigma^{\varepsilon'/3}(\rho_X|\rho_Y, \theta, F, \delta) \leq \mdim_\Sigma^{\varepsilon'/3}(\rho_X|\rho_Y) +\beta.
\end{align}

Since $\rho'_X$ is dynamically generating, by Lemma \ref{Map pseudo},  there exist  $F' \in  \cF(\Gamma)$ and $\delta' > 0$ such that as a map $\sigma \colon \Gamma \to \Sym(d)$ is good enough, we have
\begin{equation}\label{Mapcompare}
\Map(\rho'_X, F', \delta', \sigma) \subseteq \Map(\rho_X, F ,\delta, \sigma).
\end{equation} 

For every $\varphi \in \Map(\rho_X, F, \delta, \sigma)$, by Lemma \ref{Map lowerbound}, the set 
\begin{equation}\label{cardi}
\{v \in [d]: \rho(s\varphi(v), \varphi(\sigma_s(v))) \leq \sqrt{\delta} {\rm \ for \ every \ } s \in K\}
\end{equation}
 has cardinality at least $(1-|K|\delta)d$. 
 Take a partition of unit $\{\zeta_U \}_{U \in \sU}$ of $X$ subordinate to $\sU$, i.e.
 \begin{itemize}
\item[i)] each map $\zeta_U \colon X \to [0, 1]$ is continuous with {\rm supp} $(\zeta_U) \subseteq U$; \\
\item[ii)] $\sum_{U \in \sU} \zeta_U(x) = 1$ for every $x \in X$.
\end{itemize}
Consequently, we obtain an $(\varepsilon, \rho'_X)$-embedding $ \overrightarrow{\zeta} \colon X \to [0, 1]^\sU$ sending $x$ to $(\zeta_U(x))_{U \in \sU}$. 

Next we define a continuous map $\overrightarrow{h} \colon \Map(\rho_X, F, \delta, \sigma) \to ([0, 1]^\sU)^d$  by
$$(\overrightarrow{h}(\varphi))_v=\overrightarrow{\zeta}(\varphi(v))\max_{s \in K} \left(\max \left(\rho_X(s\varphi(v), \varphi(\sigma_s(v)))-\sqrt{\delta}, 0\right)\right)$$
for every $v \in [d]$ and $\varphi \in \Map(\rho_X, F, \delta, \sigma)$. Denote by $\overrightarrow{0}$ the element of $ [0, 1]^\sU$ taking value zero at all coordinates.
From (\ref{cardi}), we have that the image of $\overrightarrow{h}$ falls into the closed set
$$Z:=\{\omega \colon [d] \to [0, 1]^\sU:  |\{v \in [d]: \omega(v)=\overrightarrow{0}\}| \geq (1-|K|\delta)d   \}.$$
Since $\delta |\sU||K| < \beta$, by \cite[Corollary 1.2.6]{C15}, we have $\dim (Z) \leq |\sU| |K|\delta d < \beta d$.

Now let  $\Phi \colon \Map(\rho_X, F, \delta, \sigma) \to P$ be an $(\varepsilon'/3, \theta, \rho_{X, \infty}|\rho_{Y, \infty})$-embedding for some compact metrizable space $P$ with $\dim P = \Wdim_{\varepsilon'/3}(\Map(\rho_X, F, \delta, \sigma), \theta, \rho_{X, \infty}|\rho_{Y, \infty})$. This induces a continuous map  $\Psi \colon \Map(\rho'_X, F', \delta', \sigma) \to P\times Z$ by
$$\Psi(\varphi)=(\Phi(\varphi), \overrightarrow{h}(\varphi)).$$
By (\ref{Mapcompare}) and the definition of $\overrightarrow{h}$, the map $\Psi$ is well-defined.\\

{\bf Claim:} $\Psi$ is an $(\varepsilon, \theta, \rho'_{X, \infty}|\rho_{Y, \infty})$-embedding.\\

Therefore, by  the choice of $\Phi$, we have
\begin{align*}
&\Wdim_\varepsilon(\Map(\rho'_X, F', \delta', \sigma), \theta, \rho'_{X, \infty}|\rho_{Y, \infty}) \\
& \leq \dim(P\times Z)\\
                                      & \leq \Wdim_{\varepsilon'/3}(\Map(\rho_X, F, \delta, \sigma), \theta, \rho_{X, \infty}|\rho_{Y, \infty})+\beta d
\end{align*}
for all good enough maps $\sigma \colon \Gamma \to \Sym(d)$.
By (\ref{Fdelta}) it follows that
\begin{align*}
\mdim_\Sigma^\varepsilon(\rho'_X|\rho_Y, \theta, F', \delta') &\leq \mdim_\Sigma^{\varepsilon'/3}(\rho_X|\rho_Y, \theta, F, \delta) +\beta \\
       &\leq \mdim_\Sigma^{\varepsilon'/3}(\rho_X|\rho_Y) +2\beta\\
       &\leq \mdim_\Sigma(\rho_X|\rho_Y) +2\beta
\end{align*}
and hence
$$\mdim_\Sigma^\varepsilon(\rho'_X|\rho_Y) \leq \mdim_\Sigma^\varepsilon(\rho'_X|\rho_Y, \theta, F', \delta') \leq \mdims(\rho_X|\rho_Y) +2\beta$$
as desired.\\

{\it Proof of Claim.} 
Suppose that $\varphi, \varphi' \in \Map(\rho'_X, F', \delta', \sigma)$ satisfy that $\Psi(\varphi)=\Psi(\varphi')$ and $\rho_{Y, \infty}(\pi\circ \varphi, \pi \circ \varphi') \leq \theta$. We need to show $\rho'_X(\varphi(v), \varphi'(v)) < \varepsilon$ for every $v \in [d]$.
Let us consider $v \in [d]$ in two cases.\\

{\bf Case 1.} $\overrightarrow{h}(\varphi)_v=\overrightarrow{0} \in [0, 1]^\sU$;

Since $\{\zeta_U\}_{U \in \sU}$ is a partition of unit, there exist $U, U' \in \sU$ such that $\zeta_U(\varphi(v)) \neq 0$ and $\zeta_{U'}(\varphi'(v)) \neq 0$. It forces that for every $s \in K$,
$$\rho_X(s\varphi(v), \varphi(\sigma_s(v))) \leq \sqrt{\delta} {\rm \ and \ } \rho_X(s\varphi'(v), \varphi'(\sigma_s(v)))  \leq \sqrt{\delta}.$$
Since $\Phi$ is an $(\varepsilon'/3, \theta,  \rho_{X, \infty}|\rho_{Y, \infty})$-embedding, for every $s \in K$, we have
$$\rho_X(\varphi(\sigma_s(v)), \varphi'(\sigma_s(v))) < \varepsilon'/3.$$
Note that $\sqrt{\delta} < \varepsilon'/3$. Thus for every $s \in K$ one has
\begin{align*}
\rho_X(s\varphi(v), s\varphi'(v)) & \leq \rho_X(s\varphi(v), \varphi(\sigma_s(v))) +  \rho_X(\varphi(\sigma_s(v)), \varphi'(\sigma_s(v)) )\\
                          &\ \ \  +\rho_X(\varphi'(\sigma_s(v)), s\varphi'(v)) \\
                          & \leq \sqrt{\delta} + \varepsilon'/3 +\sqrt{\delta} < \varepsilon'.
\end{align*}
By the choice of $\varepsilon'$, we conclude that $\rho'_X(\varphi(v), \varphi'(v)) < \varepsilon$.\\

{\bf Case 2.}  $\overrightarrow{h}(\varphi)_v \neq \overrightarrow{0}$.

In this case there exists $U \in \sU$ such that $\overrightarrow{h}(\varphi)_{v, U} > 0$. Since $\overrightarrow{h}(\varphi)_v =\overrightarrow{h}(\varphi')_v \neq \overrightarrow{0}$, according to the definition of $\overrightarrow{h}$, it implies that $\zeta_U(\varphi(v)) > 0$ and $\zeta_U(\varphi'(v)) > 0$.
Thus $\varphi(v), \varphi'(v) \in {\rm supp} (\zeta_U)  \subseteq U$.  Since ${\rm diam} (U, \rho'_X) < \varepsilon$, we have
$$\rho'_X(\varphi(v), \varphi'(v)) \leq {\rm diam} (U, \rho'_X) < \varepsilon.$$
This completes the proof of Claim.

\end{proof}

\begin{proposition}  \label{second reduction}
Fix a dynamically generating continuous pseudomentric $\rho_X$ on $X$. Let $\rho_Y$ and $\rho_Y'$ be two continuous pseudometrics on $Y$ which are dynamically generating. Then for any factor map $\pi \colon X \to Y$, we have
$$\mdims(\rho_X|\rho_Y)=\mdims(\rho_X|\rho_Y').$$
In particular, we have $\mdims(X|Y)=\mdims(\rho_X|\rho_Y)$.
\end{proposition}

\begin{proof}
Fix $\beta, \theta >0$. We will show $\mdim_\Sigma^\varepsilon(\rho_X|\rho_Y) \leq \mdim_\Sigma^\varepsilon(\rho_X|\rho_Y', \theta)+ 2\beta$ for every $\varepsilon > 0$. By symmetry, the conclusion then follows.

Since $\rho_Y$ is dynamically generating, there exist  $K \in \cF(\Gamma)$ and $\theta' > 0$ such that for every $y, y' \in X$ satisfying that  $ \rho_{Y, K}(y, y') < \theta'$, then one has  $\rho'_Y(y, y') < \theta$.
Let $\sV$ be a finite open cover of $Y$ with ${\rm diam}(V, \rho'_{Y}) < \theta$ for every $V \in \sV$. Take a finite set $F \subseteq \Gamma$ containing $K$,  and $\delta > 0$ with $\sqrt{\delta} < \theta'/3$ and $\delta |\sV||K| < \beta$ such that 
\begin{align}\label{FdeltaY}
\mdim_\Sigma^{\varepsilon}(\rho_X|\rho'_Y, \theta, F, \delta) \leq \mdim_\Sigma^{\varepsilon}(\rho_X|\rho'_Y, \theta) +\beta.
\end{align}

Since $\rho_X$ is dynamically generating, by Lemma \ref{Map into},  there exist  $F' \in  \cF(\Gamma)$ containing $F$ and $\delta > \delta' > 0$ such that as a map $\sigma \colon \Gamma \to \Sym(d)$ is a good enough, we have
\begin{equation}\label{MapcompareY}
\pi^d(\Map(\rho_X, F', \delta', \sigma)) \subseteq \Map(\rho_Y, F ,\delta, \sigma).
\end{equation} 

For every $\psi \in \Map(\rho_Y, F, \delta, \sigma)$, by Lemma \ref{Map lowerbound}, the set 
\begin{equation*}\label{cardiY}
\{v \in [d]: \rho_Y(s\psi(v), \psi(\sigma_s(v))) \leq \sqrt{\delta} {\rm \ for \ every \ } s \in K\}
\end{equation*}
 has cardinality at least $(1-|K|\delta)d$. 
 Taking a partition of unit $\{\zeta_V \}_{V \in \sV}$ of $Y$ subordinate to $\sV$, we obtain a $(\theta, \rho'_Y)$-embedding $ \overrightarrow{\zeta} \colon Y \to [0, 1]^\sV$ sending $y$ to $(\zeta_V(y))_{V \in \sV}$. 

Now we consider a continuous map $\overrightarrow{h} \colon \Map(\rho_Y, F, \delta, \sigma) \to ([0, 1]^\sV)^d$ defined by
$$(\overrightarrow{h}(\psi))_v=\overrightarrow{\zeta}(\psi(v))\max_{s \in K} \left(\max \left(\rho_Y(s\psi(v), \psi(\sigma_s(v)))-\sqrt{\delta}, 0\right)\right)$$
for every $v \in [d]$ and $\psi \in \Map(\rho_Y, F, \delta, \sigma)$. Denote by $\overrightarrow{0}$ the element of $ [0, 1]^\sV$ taking value zero at all coordinates.
Then the image of $\overrightarrow{h}$ falls into the closed set
$$Z:=\{\omega \colon [d] \to [0, 1]^\sV:  |\{v \in [d]: \omega(v)=\overrightarrow{0}\}| \geq (1-|K|\delta)d   \}.$$
Since $\delta |\sV||K| < \beta$, by \cite[Corollary 1.2.6]{C15}, we have
$\dim (Z) \leq |\sV| |K|\delta d < \beta d$.

Let  $\Phi \colon \Map(\rho_X, F, \delta, \sigma) \to P$ be an $(\varepsilon, \theta, \rho_{X, \infty}|\rho'_{Y, \infty})$-embedding for some compact metrizable space $P$ with $\dim P =\Wdim_{\varepsilon}(\Map(\rho_X, F, \delta, \sigma), \theta, \rho_{X, \infty}|\rho'_{Y, \infty})$. Set $\Psi \colon \Map(\rho_X, F', \delta', \sigma) \to P\times Z$ by
$$\Psi(\varphi)=(\Phi(\varphi), \overrightarrow{h}(\pi \circ \varphi))$$
for every $\varphi \in \Map(\rho_X, F', \delta', \sigma)$.
By (\ref{MapcompareY}) and the definition of $\overrightarrow{h}$, the map $\Psi$ is well-defined and continuous.\\

{\bf Claim:} $\Psi$ is an $(\varepsilon, \theta'/3, \rho_{X, \infty}|\rho_{Y, \infty})$-embedding.\\

Therefore, by the choice of $\Phi$, we have
\begin{align*}
&\Wdim_\varepsilon(\Map(\rho_X, F', \delta', \sigma), \theta'/3, \rho_{X, \infty}|\rho_{Y, \infty}) \\
& \leq \dim(P\times Z)\\
& \leq \Wdim_{\varepsilon}(\Map(\rho_X, F, \delta, \sigma), \theta, \rho_{X, \infty}|\rho'_{Y, \infty})+\beta d
\end{align*}
for all good enough maps $\sigma \colon \Gamma \to \Sym(d)$.
By (\ref{FdeltaY}) it follows that
\begin{align*}
\mdim_\Sigma^\varepsilon(\rho_X|\rho_Y, \theta'/3, F', \delta') &\leq \mdim_\Sigma^{\varepsilon}(\rho_X|\rho'_Y, \theta, F, \delta) +\beta \\
       &\leq \mdim_\Sigma^{\varepsilon}(\rho_X|\rho'_Y, \theta) +2\beta
\end{align*}
and hence
$$\mdim_\Sigma^\varepsilon(\rho_X|\rho_Y) \leq \mdim_\Sigma^\varepsilon(\rho_X|\rho_Y, \theta'/3, F', \delta') \leq \mdim_\Sigma^\varepsilon(\rho_X|\rho'_Y, \theta) +2\beta$$
as desired.\\

{\it Proof of Claim.} 
Suppose that $\varphi, \varphi' \in \Map(\rho_X, F', \delta', \sigma)$ satisfy that $\Psi(\varphi)=\Psi(\varphi')$ and $\rho_{Y, \infty}(\pi\circ \varphi, \pi \circ \varphi') \leq \theta'/3$. We need to show $\rho_{X, \infty}(\varphi, \varphi') < \varepsilon$. Since $\Phi$ is an $(\varepsilon, \theta, \rho_{X, \infty}|\rho'_{Y, \infty})$-embedding, it suffices to show $\rho'_{Y, \infty}(\pi\circ \varphi, \pi \circ \varphi') \leq \theta$.
Let us consider $v \in [d]$ in the following two cases.\\

{\bf Case 1.} $\overrightarrow{h}(\pi\varphi)_v=\overrightarrow{0} \in [0, 1]^\sV$;

Since $\{\zeta_V\}_{V \in \sV}$ is a partition of unit, there exist $V, V' \in \sV$ such that $\zeta_V(\pi\varphi(v)) > 0$ and $\zeta_{V'}(\pi\varphi'(v)) > 0$. From the definition of $\overrightarrow{h}$,  it forces that for every $s \in K$,
$$\rho_Y(s\pi\varphi(v), \pi\varphi(\sigma_s(v))) \leq \sqrt{\delta} {\rm \ and \ } \rho_Y(s\pi\varphi'(v), \pi\varphi'(\sigma_s(v))) \leq \sqrt{\delta}$$
Given that $\sqrt{\delta} < \theta'/3$,  for every $s \in K$ one has
\begin{align*}
\rho_Y(s\pi\varphi(v), s\pi\varphi'(v)) & \leq \rho_Y(s\pi\varphi(v), \pi\varphi(\sigma_s(v))) +  \rho_Y(\pi\varphi(\sigma_s(v)), \pi\varphi'(\sigma_s(v)) )\\
                          &\ \ \  +\rho_Y(\pi\varphi'(\sigma_s(v)), s\pi\varphi'(v)) \\
                          & \leq \sqrt{\delta} + \theta'/3 +\sqrt{\delta} < \theta'.
\end{align*}
By the choice of $\theta'$, it follows that $\rho'_Y(\pi\varphi(v), \pi\varphi'(v)) \leq \theta$.\\

{\bf Case 2.}  $\overrightarrow{h}(\pi\varphi)_v \neq \overrightarrow{0}$.

In this case there exists $V \in \sU$ such that $\overrightarrow{h}(\pi\varphi)_{v, V} > 0$. Since $\overrightarrow{h}(\pi\varphi)_v =\overrightarrow{h}(\pi\varphi')_v \neq \overrightarrow{0}$, according to the definition of $\overrightarrow{h}$, it implies that $\zeta_V(\pi\varphi(v)) > 0$ and $\zeta_V(\pi\varphi'(v)) > 0$.
Thus both $\pi\varphi(v)$ and $\pi\varphi'(v)$ reside within  ${\rm supp} (\zeta_V)  \subseteq V$.  Given that ${\rm diam} (V, \rho'_Y) < \theta$, we have
$$\rho'_Y(\pi\varphi(v), \pi\varphi'(v)) \leq {\rm diam} (V, \rho'_Y) < \theta.$$
This completes the proof of Claim.

\end{proof}

By adopting the argument of \cite[Theorem 3.1]{Li13}, we prove that the sofic conditional mean dimension generalizes the notion of conditional mean dimension. The proof breaks down in Lemmas \ref{mdimS larger} and \ref{mdimS smaller}.
\begin{theorem} \label{conditional extension}
Let $\pi \colon X \to Y$ be a factor map under the actions of an infinite amenable group $\Gamma$.
Then
$\mdim_\Sigma(X|Y) = \mdim(X|Y)$.
\end{theorem}

\begin{lemma} \label{mdimS larger}
Let $\pi \colon X \to Y$ be a factor map under the actions of an infinite amenable group $\Gamma$.
Let $\rho_X, \rho_Y$ be continuous pseudometrics on $X$ and $Y$ respectively. Then for every $\varepsilon >0$ we have $\mdim_{\Sigma}^\varepsilon(\rho_X|\rho_Y) \geq \mdim_\varepsilon(\rho_X|\rho_Y)$. In particular,
$\mdim_\Sigma(X|Y) \geq \mdim(X|Y)$.
\end{lemma}

\begin{proof}
Without loss of generality we may assume that $\mdim_\varepsilon(\rho_X|\rho_Y) > 0$.  Fix $0< \beta < \mdim_\varepsilon(\rho_X|\rho_Y)$. We aim to prove 
$$\mdim_\Sigma^\varepsilon(\rho_X|\rho_Y, \theta) \geq \mdim_\varepsilon(\rho_X|\rho_Y) -2\beta$$
as $\theta > 0$ satisfies that $ \mdim_\varepsilon(\rho_X|\rho_Y, \theta) -\beta >0$.
Let  $F  \in \cF(\Gamma)$ and $\delta > 0$. It suffices to show that as $\sigma$ is good enough, one has
$$\Wdim_\varepsilon(\Map(\rho_X, F, \delta, \sigma), \theta, \rho_{X, \infty}|\rho_{Y, \infty}) \geq (\mdim_\varepsilon(\rho_X|\rho_Y, \theta) - 2 \beta)d.$$

Take a finite set $K \subseteq \Gamma$ containing $F$ and $\varepsilon_0 > 0$ such that for any  $F' \in  \cF(\Gamma)$ with $|KF'\setminus F'| < \varepsilon_0|F'|$, one has
\begin{equation} \label{lowerF}
\Wdim_\varepsilon(\rho_{X, F'}|\rho_{Y, F'}, \theta) \geq (\mdim_\varepsilon(\rho_X|\rho_Y, \theta) - \beta)|F'|.
\end{equation}
Since $ \mdim_\varepsilon(\rho_X|\rho_Y, \theta) -\beta >0$, we can find $ \tau > 0$ satisfying $\sqrt{\tau} {\rm diam}(X, \rho_X) < \delta/2$ and 
\begin{equation} \label{tau}
(1 - \tau)(\mdim_\varepsilon(\rho_X|\rho_Y, \theta) -\beta) \geq \mdim_\varepsilon(\rho_X|\rho_Y, \theta) -2\beta.
\end{equation} 

Applying Lemma \ref{amenable generalize}, there exist $\ell \in \mathbb{N}$ and finite $F_1, \dots, F_\ell \subseteq \Gamma$ with $|KF_k \setminus F_k| < \min ( \varepsilon_0, \tau) |F_k|$ for every $k=1, \dots, \ell$ such that for any good enough sofic approximation map $\sigma \colon \Gamma \to \Sym(d)$ and any $\mathcal{W} \subseteq [d]$ with $|\mathcal{W}| \geq (1-\tau/2)d$, there exist $\mathcal{C}_1, \dots, \mathcal{C}_\ell \subseteq \cW$ satisfying the following: 
\begin{itemize}
\item[i)] the map $F_k \times \cC_k \to \sigma(F_k)\cC_k$ sending $(s, c)$ to $\sigma_s(c)$ is bijective for every $k =1, \dots, \ell$;\\
\item[ii)] the sets $\{\sigma(F_k)\cC_k\}_{k=1}^\ell$ are pairwise disjoint with $|\bigsqcup_{k=1}^\ell \sigma(F_k)\cC_k| \geq (1-\tau)d$.
\end{itemize}

Let $\sigma \colon \Gamma \to \Sym(d)$ be a map. Consider the subset
$$\cW =\{v \in  [d]:   \sigma_t\sigma_s(v)=\sigma_{ts}(v) {\rm \ for \ every \ } t \in F, s \in \cup_{k=1}^\ell F_k\}.$$
 By the soficity of $\Gamma$, as $\sigma$ is a good enough sofic approximation, we have $|\cW| \geq (1-\tau/2)d$ and there exist $\cC_1, \dots, \cC_\ell \subseteq \cW$ as above.  
 
 Since $\Gamma$ is infinite, there exist maps $\psi_k  \colon \cC_k \to \Gamma$ for $k = 1, \dots, \ell$ such that the map $\Psi \colon \bigsqcup_{k=1}^\ell F_k \times \cC_k \to \Gamma$ sending $(s, c) \in F_k \times \cC_k $ to $s\psi_k(c)$ is injective. Denote by $\tilde{F}$ the image of $\Psi$. Since $|KF_k \setminus F_k| < \varepsilon_0 |F_k|$ for every $k=1, \dots, \ell$, we have that 
 \begin{equation} \label{tildeF}
 |K\tilde{F} \setminus \tilde{F}| < \varepsilon_0 |\tilde{F}| {\rm \ and  \ } |\tilde{F}| =|\bigsqcup_{k=1}^\ell \sigma(F_k)\cC_k| \geq (1-\tau)d.
 \end{equation} 

Fix $x_0 \in X$. For every $x \in X$ define  a map $\varphi_x \colon [d] \to X$ by $\varphi_x(v) =x_0$ if $v \in [d] \setminus \bigsqcup_{k=1}^\ell \sigma(F_k)\cC_k$, and 
$$\varphi_x(v)=s\psi_k(c)x$$
if $v=\sigma_s(c)$ for some $k =1, \dots, \ell$, $s \in F_k$, and $c \in \cC_k$.
In light of \cite[Lemma 3.4]{Li13}, one has $\varphi_x \in \Map(\rho_X, F, \delta, \sigma)$.

Let $\Phi  \colon \Map(\rho_X, F, \delta, \sigma) \to P$ be an $(\varepsilon, \theta, \rho_{X, \infty}|\rho_{Y, \infty})$-embedding. Since the map $X \to \Map(\rho_X, F, \delta, \sigma)$ sending $x$ to $\varphi_x$ is continuous,  the map $f \colon X \to P$ sending $x$ to $\Phi(\varphi_x)$ is also continuous.\\

{\bf Claim:} $f$ is an $(\varepsilon, \theta, \rho_{X, \tilde{F}}|\rho_{Y, \tilde{F}})$-embedding.\\

Therefore,  combining (\ref{lowerF}), (\ref{tau}) with (\ref{tildeF}), we obtain
\begin{align*}
&\Wdim_\varepsilon(\Map(\rho_X, F, \delta, \sigma), \theta, \rho_{X, \infty}|\rho_{Y, \infty}) \\
&\geq \Wdim_\varepsilon(\rho_{X, \tilde{F}}|\rho_{X, \tilde{F}}, \theta) \\
                                & \geq |\tilde{F}| (\mdim_\varepsilon(\rho_X|\rho_Y, \theta) -\beta) \\
                                & \geq (1-\tau)d(\mdim_\varepsilon(\rho_X|\rho_Y, \theta) -\beta) \\
                                &\geq d(\mdim_\varepsilon(\rho_X|\rho_Y, \theta) -2\beta)
\end{align*}
as desired.\\

{\it Proof of Claim:}
Suppose that $f(x)=f(x')$ and $\rho_{Y, \tilde{F}}(\pi(x), \pi(x')) \leq \theta$. We need to show  $\rho_X(tx, tx') < \varepsilon$ for every $t \in \tilde{F}$.

Since $\rho_{Y, \tilde{F}}(\pi(x), \pi(x')) \leq \theta$, by the construction of $\varphi_x$ and $\varphi_{x'}$, we have 
$$\rho_{Y, \infty}(\pi\circ \varphi_x, \pi \circ \varphi_{x'}) \leq \theta.$$
Since $\Phi  \colon \Map(\rho_X, F, \delta, \sigma) \to P$ is an $(\varepsilon, \theta,  \rho_{X, \infty}|\rho_{Y, \infty})$-embedding, from $f(x)=f(x')$,  we have 
$$\rho_{X, \infty}(\varphi_x, \varphi_{x'}) < \varepsilon.$$

For every $t \in \tilde{F}$ write $t=s\psi_k(c)$ for some  $k=1,\dots, \ell, s \in F_k$, and $c \in \cC_k$. Setting $v=\sigma_s(c)$, we have $\varphi_x(v)=tx$ and $\varphi_{x'}(v)=tx'$. Thus
$$\rho_X(tx, tx') =\rho_X(\varphi_x(v), \varphi_{x'}(v)) < \varepsilon$$
as desired.

\end{proof}

\begin{lemma} \label{mdimS smaller}
Let $\pi \colon X \to Y$ be a factor map under the actions of an amenable group $\Gamma$.
Let $\rho_X, \rho_Y$ be compatible metrics on $X$ and $Y$ respectively. Then for every $\varepsilon > 0$, we have 
$$\mdim_\Sigma^\varepsilon(\rho_X|\rho_Y) \leq \mdim_{\varepsilon/3}(\rho_X|\rho_Y).$$
In particular, $\mdim_\Sigma(\rho_X|\rho_Y) \leq \mdim(\rho_X|\rho_Y)$ and hence
$\mdim_\Sigma(X|Y) \leq \mdim(X|Y)$.
\end{lemma}

\begin{proof}
Fix $\beta, \varepsilon, \theta > 0$. Our goal is to show
\begin{equation} \label{mdimapprox}
 \mdim_\Sigma^\varepsilon(\rho_X|\rho_Y) \leq \mdim_\varepsilon(\rho_X|\rho_Y, \theta) + 3\beta.    
\end{equation}

Take a finite open cover $\cU$ of $X$ satisfying that ${\rm diam}(U, \rho_X) < \varepsilon$ for each $U \in \cU$. Pick $\varepsilon_0 > 0$ and  $K \in  \cF(\Gamma)$ such that for any $F' \in \cF(\Gamma)$ with $|KF' \setminus F'| < \varepsilon_0 |F'|$ we have
$$\Wdim_{\varepsilon/3}(\rho_{X, F'}|\rho_{Y, F'}, \theta) \leq (\mdim_{\varepsilon/3}(\rho_X|\rho_Y, \theta) +\beta)|F'|.$$

Take $\tau > 0$ with  $\tau\Wdim_\varepsilon(\rho_X) \leq \beta$. 
Applying Lemma \ref{amenable generalize}, there exist $\ell \in \mathbb{N}$ and finite $F_1, \dots, F_\ell \subseteq \Gamma$ with $|KF_k \setminus F_k| < \varepsilon_0 |F_k|$ for every $k=1, \dots, \ell$ such that for any good enough sofic approximation $\sigma \colon \Gamma \to \Sym(d)$ and any $\mathcal{W} \subseteq [d]$ with $|\mathcal{W}| \geq (1-\tau/2)d$, there exist $\mathcal{C}_1, \dots, \mathcal{C}_\ell \subseteq \cW$ satisfying the following: 
\begin{itemize}
\item[i)] the map $F_k \times \cC_k \to \sigma(F_k)\cC_k$ sending $(s, c)$ to $\sigma_s(c)$ is bijective for every $k =1, \dots, \ell$;\\
\item[ii)] the sets $\{\sigma(F_k)\cC_k\}_{k=1}^\ell$ are pairwise disjoint with $|\bigsqcup_{k=1}^\ell \sigma(F_k)\cC_k| \geq (1-\tau)d$.
\end{itemize}

 Set $F=\cup_{k=1}^\ell F_k$. Take $\kappa > 0$ such that for any $x, x' \in X$ with $\rho_X(x, x') \leq \kappa$ one has  for every $k=1, \dots, \ell$,
  $$\rho_{X, F_k}(x, x') < \varepsilon/3.$$
Simultaneously, choose $\theta' > 0$ such that for any $y, y' \in Y$ with $\rho_Y(y, y') \leq \theta'$, one has  for every $k=1, \dots, \ell$,
$$\rho_{Y, F_k}(y, y') \leq \theta.$$

Take $\delta > 0$ such that $\delta < \kappa^2$ and $\delta |\cU||F| < \beta$. It then reduces to show that as $\sigma$ is  good enough, the following holds:
\begin{equation} \label{reduction}
\Wdim_\varepsilon(\Map(\rho_X, F^{-1}, \delta, \sigma), \theta', \rho_{X, \infty}|\rho_{Y, \infty}) \leq (\mdim_{\varepsilon/3}(\rho_X|\rho_Y, \theta) + 3\beta)d.
\end{equation}

Put $\cW =\{v \in  [d]:   \sigma_{s^{-1}}\sigma_s(v)=\sigma_{e_\Gamma}(v)=v {\rm \ for \ every \ } s \in F\}$. By the soficity of $\Gamma$, as $\sigma$ is good enough, we have $|\cW| \geq (1-\tau/2)$ and there exist $\cC_1, \dots, \cC_\ell \subseteq \cW$ as above.  Setting $\cZ=[d] \setminus \bigsqcup_{k=1}^\ell \sigma(F_k)\cC_k$, it follows that  $|\cZ| \leq \tau d$. For every $\varphi \in \Map(\rho_X, F^{-1}, \delta, \sigma)$, according to Lemma \ref{Map lowerbound}, the set 
$$\{v \in [d]: \rho(s\varphi(v), \varphi(\sigma_s(v))) \leq \sqrt{\delta} {\rm \ for \ every \ } s \in F^{-1}\}$$
 has cardinality at least $(1-|F|\delta)d$. 
 Take a partition of unit $\{\zeta_U \}_{U \in \sU}$ of $X$ subordinate to $\sU$.
As a consequence, we obtain an $(\varepsilon, \rho_X)$-embedding $ \overrightarrow{\zeta} \colon X \to [0, 1]^\sU$ sending $x$ to $(\zeta_U(x))_{U \in \sU}$. 

Now  consider a continuous map $\overrightarrow{h} \colon \Map(\rho_X, F^{-1}, \delta, \sigma) \to ([0, 1]^\sU)^d$ defined by
$$(\overrightarrow{h}(\varphi))_v=\overrightarrow{\zeta}(\varphi(v))\max_{s \in F^{-1}} \left(\max \left(\rho_X(s\varphi(v), \varphi(\sigma_s(v)))-\kappa, 0\right)\right)$$
for every $v \in [d]$ and $\varphi \in \Map(\rho_X, F^{-1}, \delta, \sigma)$. 
Then the image of $\overrightarrow{h}$ falls into the closed set
$$Z:=\{\omega \colon [d] \to [0, 1]^\sU:  |\{v \in [d]: \omega(v)=\overrightarrow{0}\}| \geq (1-|F|\delta)d   \}.$$
Given that $\delta |\sU||F| < \beta$, by \cite[Corollary 1.2.6]{C15}, we have
$\dim (Z) \leq |\sU| |F|\delta d < \beta d$.

For each $k=1, \dots, \ell$,  choose an $(\varepsilon/3, \theta, \rho_{X, F_k}|\rho_{Y, F_k})$-embedding $f_k \colon X \to P_k$ and an $(\varepsilon, \rho_X)$-embedding  $g \colon X \to Q$ satisfying that \begin{equation}\label{WdimFk}
\dim (P_k)=\Wdim_{\varepsilon/3}(\rho_{X, F_k}|\rho_{Y, F_k}, \theta)  \rm{\ \ and \ \ } \dim (Q)=\Wdim_\varepsilon(\rho_X).
\end{equation}

Next we define a continuous map $\Phi \colon \Map(\rho_X, F^{-1}, \delta, \sigma) \to Z \times (\Pi_{k=1}^\ell P_k^{\cC_k})\times Q^\cZ$ by 
$$\Phi(\varphi):=(\overrightarrow{h}(\varphi), (f_k(\varphi(c)))_{1 \leq k \leq \ell, c \in \cC_k}, (g(\varphi(v)))_{v \in \cZ}).$$

{\bf Claim:} $\Phi$ is an $(\varepsilon, \theta', \rho_{X, \infty}|\rho_{Y, \infty})$-embedding.\\

Therefore, by combining (\ref{mdimapprox}) with (\ref{WdimFk}),  we have
\begin{align*}
&\Wdim_\varepsilon(\Map(\rho_X, F^{-1}, \delta, \sigma), \theta', \rho_{X, \infty}|\rho_{Y, \infty})\\
&\leq  \dim \left(Z \times (\Pi_{k=1}^\ell P_k^{\cC_k})\times Q^\cZ
\right)\\
     &\leq \dim(Z) +\sum_{k=1}^\ell |\cC_k|\dim(P_k) +|\cZ|\dim(Q) \\
    & \leq \beta d + \sum_{k=1}^\ell |\cC_k|\Wdim_{\varepsilon/3}(\rho_{X, F_k}|\rho_{Y, F_k}, \theta) +\tau d\Wdim_\varepsilon(\rho_X) \\
    & \leq \beta d + \sum_{k=1}^\ell |\cC_k||F_k|(\mdim_{\varepsilon/3}(\rho_X|\rho_Y, \theta)+\beta) +\beta d \\
    &\leq \beta d +d(\mdim_{\varepsilon/3}(\rho_X|\rho_Y, \theta)+\beta) +\beta d.
\end{align*}

Hence  the inequality (\ref{reduction}) follows.\\

{\it Proof of Claim:}
Suppose that $\varphi, \varphi' \in \Map(\rho_X, F^{-1}, \delta, \sigma)$ satisfy that $\Phi(\varphi)=\Phi(\varphi')$ and $\rho_{Y, \infty}(\pi\circ \varphi, \pi\circ \varphi') \leq \theta'$, we need to verify that  for every $v \in [d]$,
\begin{equation} \label{verify}
\rho_X(\varphi(v), \varphi'(v)) < \varepsilon
\end{equation}
 For $v \in \cZ$, since $g(\varphi(v))=g(\varphi'(v))$,  by the definition of $g$, it follows that (\ref{verify}) holds.
For $v \notin \cZ$, we divide the discussion into the following two cases.\\

{\bf Case 1.}  $\overrightarrow{h}(\varphi)_v \neq \overrightarrow{0} \in [0, 1]^\sU$;\\

By the construction of $\overrightarrow{\zeta}$, we have $\overrightarrow{h}(\varphi)_{v, U} \neq 0$ for some $U \in \sU$.
It follows that $\varphi(v), \varphi'(v) \in U$. Since ${\rm diam} (U, \rho_X) < \varepsilon$, the inequality (\ref{verify}) is verified.\\

{\bf Case 2.}  $\overrightarrow{h}(\varphi)_v = \overrightarrow{0}$.\\

In this case, we write $v=\sigma_s(c)$ for the unique $1 \leq k \leq \ell$, $s \in F_k$, and $c \in \cC_k$.
Given that $\rho_Y(\pi(\varphi(c)), \pi(\varphi'(c))) \leq \theta'$, by the choice of $\theta'$, we deduce that
$$\rho_{Y, F_k}(\pi(\varphi(c)), \pi(\varphi'(c))) \leq \theta.$$
Since $f_k$ is an $(\varepsilon/3, \theta, \rho_{X, F_k}|\rho_{Y, F_k})$-embedding, from $f_k(\varphi(c))=f_k(\varphi'(c))$, we obtain
$\rho_X(s\varphi(c), s\varphi'(c)) < \varepsilon/3.$
On the other hand, since  $\{\zeta_U\}_{U \in \sU}$ is a partition of unit, there exists $U \in \sU$ such that $\zeta_U(\varphi(v)) \neq 0$. From the design of $\overrightarrow{h}$ it forces that
$$\rho_X(s^{-1}\varphi(v), \varphi(\sigma_{s^{-1}}(v))) \leq \kappa {\rm \ and \ } \rho_X(s^{-1}\varphi'(v), \varphi'(\sigma_{s^{-1}}(v))) \leq \kappa.$$
From $\cC_k \subseteq \cW$, we have 
 $$\sigma_{s^{-1}}(v)=\sigma_{s^{-1}}(v)\sigma_s(c)=\sigma_{e_\Gamma}(c)=c.$$
Thus
$\rho_X(s^{-1}\varphi(v), \varphi(c))  \leq \kappa$. 
By the choice of $\kappa$, we obtain $\rho_X(\varphi(v), s\varphi(c)) < \varepsilon/3$. Simultaneously, $\rho_X(\varphi'(v), s\varphi'(c)) < \varepsilon/3$ also holds.
Therefore,
\begin{align*}
\rho_X(\varphi(v), \varphi'(v)) &\leq \rho_X(\varphi(v), s\varphi(c)) +\rho_X(s\varphi(c), s\varphi'(c)) +\rho_X(s\varphi'(c), \varphi'(v)) \\
                              & < \varepsilon/3 +\varepsilon/3 + \varepsilon/3 =\varepsilon.
\end{align*}

\end{proof}

In the context of sofic group actions, we can establish an estimate for mean dimension of an extension system as follows:

\begin{theorem} \label{general upperbound}
For any factor map $\pi \colon X \to Y$, we have
$$\mdim_\Sigma(X) \leq \mdim_\Sigma(X|Y) +\mdims(Y|X).$$
\end{theorem}

\begin{proof}
Let $\rho_X$ and $\rho_Y$ be compatible metrics on $X$ and $Y$ respectively. Fix $\varepsilon, \theta, \delta >0$ and  $F \in \cF(\Gamma)$. 
For every map $\sigma \colon \Gamma \to \Sym(d)$ choose an $(\varepsilon, \theta, \rho_{X, \infty}|\rho_{Y, \infty})$-embedding $\Phi \colon \Map(\rho_X, F, \delta, \sigma) \to P$, as well as a $(\theta, \rho_{Y, \infty})$-embedding $\Psi \colon \pi^d(\Map(\rho_X, F, \delta, \sigma)) \to Q$. Then the continuous map 
$$ \Map(\rho_X, F, \delta, \sigma) \to P\times Q$$
sending $\varphi$ to $(\Phi(\varphi), \Psi(\pi \circ \varphi))$ is an $(\varepsilon, \rho_{X, \infty})$-embedding. It implies that
\begin{align*}
\Wdim_\varepsilon(\Map(\rho_X, F, \delta, \sigma), \rho_{X, \infty}) &\leq \Wdim_\varepsilon(\Map(\rho_X, F, \delta, \sigma), \rho_{X, \infty}|\rho_{Y, \infty}, \theta)\\
& \ \ \ +\Wdim_\theta(\pi^d(\Map(\rho_X, F, \delta, \sigma)), \rho_{Y, \infty})
\end{align*}
and hence 
$$\mdim_{\Sigma}^\varepsilon(\rho_X, F, \delta) \leq \mdim_{\Sigma}^\varepsilon(\rho_X|\rho_Y, \theta, F, \delta) +\mdim_{\Sigma}^\theta(\rho_Y|\rho_X, F, \delta).$$
Taking the infimum over $\delta > 0$ and $F \in \cF(\Gamma)$, we obtain
\begin{align*}
\mdim_{\Sigma}^\varepsilon(\rho_X) & \leq \mdim_{\Sigma}^\varepsilon(\rho_X|\rho_Y, \theta) +\mdim_{\Sigma}^\theta(\rho_Y|\rho_X).
\end{align*}
Letting $\theta $ tend to zero we obtain
$$\mdim_{\Sigma}^\varepsilon(\rho_X) \leq \mdim_{\Sigma}^\varepsilon(\rho_X|\rho_Y) +\mdims(\rho_Y|\rho_X).$$
As $\varepsilon > 0$ is arbitrary, the desired estimation then follows.
\end{proof}

To end this section, we generalize the relative mean dimension introduced by Tsukamoto \cite{T22} into the sofic context.  Recall that as $\Gamma$ is an amenable group and $\rho_X$ is a compatible metric on $X$, the {\it relative mean dimension of a factor map $\pi \colon X \to Y$} is defined as
$$\mdim(\pi):=\sup_{\varepsilon > 0} \lim_F \frac{\Wdim_\varepsilon(\pi^{-1}(y), \rho_{X, F})}{|F|}, $$
where the existence of the above limit is guaranteed by Lemma \ref{OW}.  

\begin{definition} \label{relative smdim}
  Let $\Gamma$ be a sofic group. For any factor map $\pi \colon X \to Y$, we define the {\it relative sofic mean dimension}  of $\pi$  as
    $$\mdims(\pi)=\sup_{\varepsilon > 0}  \inf_{F \in \cF(\Gamma)} \inf_{\delta > 0} \limsup_{i \to \infty} \frac{\sup_{\psi \in Y^{d_i}}\Wdim_\varepsilon(\Map(\rho_X, F, \delta, \sigma_i)\cap ((\pi^{d_i})^{-1}(\psi)), \rho_{X, \infty})}{d_i}$$
       When $\varepsilon, F, \delta$ are fixed, we write $\mdim_\Sigma^\varepsilon(\pi, F, \delta)$ for the above limit supremum. In a similar situation, we have the notions $\mdim_\Sigma^\varepsilon(\pi)$ etc. If  $\Map(\rho_X, F, \delta, \sigma_i)$ is empty for all sufficiently large $i$, the limit supremum is set to be $-\infty$.

       Again the relative sofic mean dimension is independent of the choice of compatible metric $\rho_X$. Moreover, it can be computed using dynamically generating continuous pseudometrics following from the preceding discussion in Proposition \ref{dynamical reduction}.
\end{definition}

Intuitively, $\mdims(\pi)$ describes the mean dimension of fibers individually, in contrast to $\mdims(X|Y)$. Observe that $\mdims(\pi)$ serves as a natural lower bound of $\mdims(X|Y)$ as illustrated below.  
\begin{proposition}
    For any factor map $\pi \colon X \to Y$, one has $\mdims(\pi) \leq \mdims(X|Y)$.
\end{proposition}

\begin{proof}
  Let $\rho_X$ be a compatible metric on $X$.  From the definition of relative sofic mean dimension, one can easily deduce that 
    $\mdims(\pi) \leq \mdims(\rho_X|Y) $
    and hence the conclusion follows.
\end{proof}

The proof of the following generalizing result uses a similar technique as in proving Theorem \ref{conditional extension}, whose details we omit.
\begin{theorem}
    If $\Gamma$ is an infinite amenable group, then $\mdims(\pi)=\mdim(\pi)$.
\end{theorem}
\section{Conditional sofic mean dimension of $G$-extensions}
\numberwithin{equation}{section}
\setcounter{equation}{0}
In this section, under some technical assumptions restricted from the definitions of sofic invariants, we do some computations on the conditional sofic mean dimension and relative sofic mean dimension. In particular, we obtain the results for factor maps of $G$-extensions.

Let us start with a simple case of diagonal actions on product spaces.
\begin{proposition} \label{product case}
Let $\Gamma$ continuously act on two compact metrizable spaces $Y$ and $Z$. Consider the diagonal action of $\Gamma$ on $X:=Y\times Z$ and the projection map $\pi \colon X \to Y$ as the factor map. Then
$\mdims(X|Y)\leq \mdims(Z)$; If $\mdims(Y)\geq 0$, we have $\mdims(X|Y)=\mdims(Z)$.
\end{proposition}

\begin{proof}
Fix two compatible metrics $\rho_Y, \rho_Z$ on $Y$ and $Z$ respectively. Consider the compatible metric $\rho_X$ on $X$ defined by
$$\rho_X((y, z), (y', z')):=\max(\rho_Y(y, y'), \rho_Z(z, z')).$$
Fix $\varepsilon, \theta, \delta > 0$ and $F \in \cF(\Gamma)$. To show $\mdims(X|Y)\leq \mdims(Z)$, it suffices to show as $\theta < \varepsilon$, 
\begin{equation} \label{product upper bound}
    \Wdim_\varepsilon(\Map(\rho_X, F, \delta, \sigma), \theta, \rho_{X, \infty}|\rho_{Y, \infty}) \leq \Wdim_\varepsilon(\Map(\rho_Z, F, \delta, \sigma), \rho_{X, \infty})
\end{equation}
holds for all such $\delta, F$ and any map $\sigma \colon \Gamma  \to \Sym(d)$.

Let $\Phi \colon \Map(\rho_Z, F, \delta, \sigma) \to P$ be an $(\varepsilon, \rho_{Z, \infty})$-embedding. Denote by $\pi'$ the projection map from $X=Y \times Z$ to $Z$. Note that $\rho_Z ( \pi'(x), \pi'(x')) \leq \rho_X(x, x')$ for all $x, x' \in X$. It follows that $\pi'
\circ \varphi \in \Map(\rho_Z, F, \delta, \sigma)$ for every $\varphi \in \Map(\rho_X, F, \delta, \sigma)$. Thus the map 
$\Psi \colon \Map(\rho_X, F, \delta, \sigma) \to P$ sending $\varphi$ to $\Phi(\pi' \circ \varphi)$ is well defined and continuous.

To establish the inequality \eqref{product upper bound}, it is sufficient to verify that $\Psi$ is an $(\varepsilon, \theta, \rho_{X, \infty}|\rho_{Y, \infty})$-embedding. Assume $\Psi(\varphi)=\Psi(\varphi')$ and $\rho_{Y, \infty}(\pi\circ \varphi, \pi \circ \varphi') \leq \theta$ for some $\varphi, \varphi' \in \Map(\rho_X, F, \delta, \sigma)$. From $\Psi(\varphi)=\Psi(\varphi')$, we obtain $\rho_{Z, \infty}(\pi'
\circ \varphi, \pi' \circ \varphi') < \varepsilon$. Given that $\theta < \varepsilon$, we have
$$\rho_{X, \infty}(\varphi, \varphi')=\max(\rho_{Y, \infty}(\pi \circ \varphi, \pi \circ \varphi'), \rho_{Z, \infty}(\pi'
\circ \varphi, \pi' \circ \varphi')) < \varepsilon$$
and hence $\Psi$ is an $(\varepsilon, \theta, \rho_{X, \infty}|\rho_{Y, \infty})$-embedding.

Now suppose $\mdims(Y)\geq 0$. It remains to show $\mdims(X|Y)\geq \mdims(Z)$. Fix $\varepsilon, \theta, \delta > 0$ and $F \in \cF(\Gamma)$. It suffices to show
\begin{align} \label{product lower bound}
\Wdim_\varepsilon(\Map(\rho_Z, F, \delta/2, \sigma), \rho_{Z, \infty}) \leq \Wdim_\varepsilon(\Map(\rho_X, F, \delta, \sigma), \theta, \rho_{X, \infty}|\rho_{Y, \infty})
\end{align}
as $\sigma$ is good enough.

Since $\mdims(Y) \geq 0$, as $\sigma$ is good enough, there exists $\psi_0 \in \Map(\rho_Y, F, \delta/2, \sigma)$ and we shall fix such an element. Let $\Phi \colon \Map(\rho_X, F, \delta, \sigma) \to P$ be an $(\varepsilon, \theta, \rho_{X, \infty}|\rho_{Y, \infty})$-embedding. Define a continuous map $\Psi \colon \Map(\rho_Z, F, \delta/2, \sigma) \to P$ sending $\varphi$ to $\Phi\circ \Xi (\varphi)$, where the map $\Xi \colon Z^{[d]} \to X^{[d]}=Y^{[d]}\times Z^{[d]}$ sends $\varphi$ to $(\psi_0, \varphi)$. To ensure that $\Psi$ is well defined, we first verify that $\Xi(\varphi) \in \Map(\rho_X, F, \delta, \sigma)$. Note that  for any $x, x' \in X$, we have
$$\rho_X^2(x, x') \leq \rho_Y^2(\pi(x), \pi(x'))+\rho_Z^2( \pi'(x), \pi'(x')).$$
Thus for every $\varphi \in \Map(\rho_Z, F, \delta/2, \sigma)$ and $s \in F$, we have
\begin{align*}
\rho_{X, 2}^2(s\Xi(\varphi), \Xi(\varphi)\circ \sigma_s) &\leq \rho_{Y, 2}^2(s\psi_0, \psi_0 \circ \sigma_s) + \rho_{Z, 2}^2(s\varphi, \varphi \circ \sigma_s) \\
& \leq (\delta/2)^2 + (\delta/2)^2 < \delta^2.
\end{align*}
It follows that $\Xi(\varphi) \in \Map(\rho_X, F, \delta, \sigma)$.

To show the inequality \eqref{product lower bound}, it suffices to verify that $\Psi$ is an $(\varepsilon, \rho_{Z, \infty})$-embedding. Let $\Psi(\varphi)=\Psi(\varphi')$ for some $\varphi, \varphi' \in \Map(\rho_Z, F, \delta/2, \sigma)$. We aim to show 
$$\rho_{Z, \infty}(\varphi, \varphi') < \varepsilon.$$ From $\pi \circ \Xi(\varphi)=\psi_0=\pi \circ\Xi(\varphi')$ we obtain 
$\rho_{Y, \infty}(\pi \circ \Xi(\varphi), \pi \circ \Xi(\varphi')) \leq \theta$. Since $\Phi$ is an $(\varepsilon, \theta, \rho_{X, \infty}|\rho_{Y, \infty})$-embedding and $\Psi(\varphi)=\Psi(\varphi')$, we have
$$\rho_{Z, \infty}(\varphi, \varphi')=\rho_{X, \infty}(\Xi(\varphi), \Xi(\varphi')) < \varepsilon$$
as desired.

\end{proof}

\begin{remark}
Note that the requirement in Proposition \ref{product case} that $\mdims(Y)\geq 0$ is necessary since there is a possibility, in general, that  $\Map(\rho_Y, F, \delta, \sigma)$ could be empty which forces that $\Map(\rho_X, F, \delta, \sigma)$ is empty as well.  However, $\Map(\rho_Z, F, \delta, \sigma)$ could be nonempty.
\end{remark}

Now we consider the broader scope of $G$-extensions (see Definition \ref{G-extension} and examples therein). It turns out that the relative mean dimension of a $G$-extension map is easier to be obtained.

\begin{proposition} \label{relative  computation}
Let  $\pi \colon X \to Y$ be a $G$-extension.  If $\mdims(X) \geq 0$, then $\mdims(\pi)=\mdims(G)$. 
\end{proposition}

The proof of Proposition \ref{relative computation} is completed in the following two lemmas.
\begin{lemma} \label{lower G-bound}
Let $\pi \colon X \to Y$ be a $G$-extension map. If $\mdims(X) \geq 0$, then $\mdims(\pi)\geq \mdims(G)$. In particular, $\mdims(X|Y)\geq \mdims(G)$.
\end{lemma}

\begin{proof}
Let $\rho_X, \rho_Y, \rho_G$ be compatible metrics on $X, Y, G$ respectively.  
 Fix $\varepsilon, \theta, \delta > 0$ and $F \in \cF(\Gamma)$. By compactness, there exists $\varepsilon' > 0$ depending only on $\varepsilon$ such that
$$\rho_G(g, g') < \varepsilon'$$  whenever $\rho_X(x_0g, x_0g') < \varepsilon$ for some $x_0 \in X$ and $g, g' \in G$.
To show $\mdims(G) \leq \mdims(X|Y)$, it reduces to show there exists $\delta' > 0$  such that
$$\Wdim_\varepsilon(\Map(\rho_G, F, \delta', \sigma), \rho_{G, \infty}) \leq \max_{\psi \in Y^d}\Wdim_{\varepsilon'}(\Map(\rho_X, F, \delta, \sigma)\cap (\pi^d)^{-1}(\psi), \rho_{X, \infty})$$
for all  good enough map $\sigma \colon \Gamma \to \Sym(d)$.
Since the map $X\times G \to X$ sending $(x, g)$ to $xg$ is uniformly continuous, there exists $\delta' > 0$ such that
$$\rho_X(xg, x'g') \leq \delta/2$$
whenever $x, x' \in X, g, g' \in G$ satisfy $\rho_X(x, x') \leq \sqrt{\delta'}$ and $\rho_G(g, g') \leq \sqrt{\delta'}$. 
We shall further require $\delta'$ is less than $\delta^2(8|F|{\rm diam}(X, \rho_X)^2)^{-1}$.

Since $\mdims(X) \geq 0$, as $\sigma$ is good enough, there exists $\varphi_0 \in \Map(\rho_X, F, \delta', \sigma)$ and we write $\psi_0=\pi\circ \varphi_0$. Let $\Phi \colon \Map(\rho_X, F, \delta, \sigma)\cap (\pi^d)^{-1}(\psi_0) \to P$ be an $(\varepsilon',  \rho_{X, \infty})$-embedding. This map then naturally induces a continuous map 
$$\Psi \colon \Map(\rho_G, F, \delta', \sigma)\cap (\pi^d)^{-1}(\psi_0) \to P$$
sending $\varphi$ to $\Phi(\widetilde{\varphi})$, where the map $\widetilde{\varphi}$ is defined via $$\widetilde{\varphi}(v)=\varphi_0(v)\varphi(v).$$
for every $v \in [d]$. By the same argument of \cite[Lemma 5.7]{Lu17} and the assumption on $\delta'$, one can verify that $\widetilde{\varphi} \in \Map(\rho_X, F, \delta, \sigma)$ and hence $\Psi$ is well defined.

The leftover is to verify that $\Psi$ is an $(\varepsilon, \rho_{G, \infty})$-embedding. Let $\Psi(\varphi)=\Psi(\varphi')$ for some $\varphi, \varphi' \in \Map(\rho_G, F, \delta',\sigma)$.  Since $\Phi$ is an $(\varepsilon', \rho_{X, \infty})$-embedding, it concludes that $\rho_{X, \infty}(\widetilde{\varphi}, \widetilde{\varphi'}) < \varepsilon'$. By our assumption on $\varepsilon'$, it implies that $\rho_{X, \infty}(\varphi, \varphi') < \varepsilon$.
\end{proof}

\begin{lemma} \label{upper G-bound}
    For any $G$-extension $\pi \colon X \to Y$, we have $\mdims(\pi) \leq \mdims(G)$.
\end{lemma}

\begin{proof}
  Let $\rho_X, \rho_Y, \rho_G$ be compatible metrics on $X, Y, G$ respectively.  Fix $\varepsilon > 0$. By compactness, there exists $\varepsilon' > 0$ such that
\begin{equation*} \label{varepsilon' assumption}
\rho_X(xg, xg') < \varepsilon
\end{equation*}
for any $x \in X$ and  $g, g' \in G$ satisfying $\rho_G(g, g') < \varepsilon'$. 
For any $\delta > 0$, by the definition of $G$-extensions, there exists $\delta' > 0$ such that $|F|\delta''D^2 \leq \delta^2/4$ and 
\begin{equation*}
    \rho_G(g, g') \leq \delta/2
\end{equation*}
for every $x, x' \in X, g, g' \in G$ with $\rho_X(x, x') \leq \sqrt{\delta'}$ and $\rho_X(xg, x'g') \leq \sqrt{\delta'}$.
Let $F \in \cF(\Gamma)$. We aim to show for every good enough map $\sigma \colon \Gamma \to \Sym(d)$ and any $\psi \in Y^d$, the following holds:
\begin{equation} \label{core estimate}
\Wdim_{\varepsilon}(\Map(\rho_X, F, \delta', \sigma)\cap (\pi^d)^{-1}(\psi), \rho_{X, \infty}) \leq
\Wdim_{\varepsilon'}(\Map(\rho_G, F, \delta, \sigma), \rho_{G, \infty})
\end{equation}

Let $\Phi \colon \Map(\rho_G, F, \delta, \sigma) \to P$ be an $(\varepsilon', \rho_{G, \infty})$-embedding.  Without loss of generalization we may assume the set $\Map(\rho_X, F, \delta', \sigma)\cap (\pi^d)^{-1}(\psi)$ is nonempty, containing an element $\varphi_0$.
For every element $\varphi \in (\pi^d)^{-1}(\psi)$, according to the definition of $G$-extensions, we can express $\varphi(v) $ as 
$$\varphi(v)=\varphi_0(v)g_{\varphi(v)}$$
for the unique $g_{\varphi(v)} \in G$. Define $\widetilde{\varphi} \colon [d] \to G$ via $\widetilde{\varphi}(v):=g_{\varphi(v)}$ for every $v \in [d]$. By the definition of $G$-extensions, the map $(\pi^d)^{-1}(\psi) \to G^{[d]}$ sending $\varphi$ to $\widetilde{\varphi}$ is continuous.

{\bf Claim.} For every $\varphi \in \Map(\rho_X, F, \delta', \sigma)\cap (\pi^d)^{-1}(\psi)$, we have $\widetilde{\varphi} \in \Map(\rho_G, F, \delta, \sigma)$;

Consequently, the map $\Psi \colon \Map(\rho_X, F, \delta', \sigma)\cap (\pi^d)^{-1}(\psi) \to P$, defined by sending $\varphi$ to $\Phi(\widetilde{\varphi})$,  is well-defined and continuous. By the choice of $\varepsilon'$, it can be readily verified that $\Psi$ is an $(\varepsilon, \rho_{X, \infty})$-embedding, from which the inequality \eqref{core estimate} then follows. 

{\it Proof of Claim.} Since $\varphi \in \Map(\rho_X, F, \delta', \sigma)$, by Lemma \ref{Map lowerbound}, the set
$$\mathcal{W}_\varphi=\{v \in [d]: \rho_X(s\varphi(v), \varphi\circ \sigma_s(v)) \leq \sqrt{\delta'} \ {\rm for \ any}  \ s \in F\}$$
has cardinality of at least $(1-|F|\delta')d$. Thus the set $\cW:=\cW_\varphi\cap \cW_{\varphi_0}$ has cardinality of at least $(1-2|F|\delta')d$.
Note that $s\varphi(v)=(s\varphi_0(v))(sg_{\varphi(v)})$ and $\varphi(\sigma_s(v))=\varphi_0(\sigma_s(v))g_{\varphi(\sigma_s(v))}$. By the choice of $\delta'$, for any $v \in \cW$, one has
$$\rho_G(sg_{\varphi(v)}, g_{\varphi(\sigma_s(v))})\leq \delta/2.$$
Set $D={\rm diam} (G, \rho_G)$. Given that $|F|\delta''D^2 \leq \delta^2/4$, it follows that
\begin{align*}
    \rho_{G, 2}^2(s\widetilde{\varphi}, \widetilde{\varphi} \circ \sigma_s) &= \frac{1}{d}\sum_{v \in \mathcal{W}} \rho_G^2(sg_{\varphi(v)}, g_{\varphi(\sigma_s(v))})+\frac{1}{d}\sum_{v \in [d]\setminus \mathcal{W}} \rho_G^2(sg_{\varphi(v)}, g_{\varphi(\sigma_s(v))})\\
    &\leq \frac{1}{d}\sum_{v \in \mathcal{W}} (\frac{\delta}{2})^2 +\frac{1}{d}\sum_{v \in [d]\setminus \mathcal{W}} D^2 \\
    & \leq \frac{1}{d}\left(d\delta^2/4 + 2|F|\delta''dD^2\right) \leq \delta^2
\end{align*}
as desired.
  
\end{proof}

 \begin{remark}
 For a $G$-extension map $\pi \colon X \to Y$ under the actions of an ameanble group,   \cite[Proposition 2.8]{L22} provides an upper bound estimation of $\mdims(X|Y)$ in terms of $\mdim(G)$ when $\pi$ admits a continuous section $\tau$. Extending this result to the sofic context reveals that, due to constraints imposed by dynamical models, a crucial condition is that $\tau$ should be $\Gamma$-equivariant.
 While examples meeting this condition can be readily constructed from cocycles, it is noteworthy that under such a robust scenario, 
the diagonal action $\Gamma \curvearrowright Y\times G$ becomes isomorphic to $\Gamma \curvearrowright X$ as dynamical systems, mapping $(y, g)$ to $\tau(y)g$. Consequently, by applying Proposition \ref{product case}, we deduce that
$$\mdims(X|Y)=\mdims(Y\times G|Y)\leq \mdims(G).$$
\end{remark}

\section{Conditional sofic metric mean dimension}
\numberwithin{equation}{section}
\setcounter{equation}{0}

In this section, we introduce the metric counterpart of the conditional mean dimension for actions of sofic groups and establish its fundamental properties. Specifically, we compute the value for factor maps of $G$-extensions and extend a result of Shi-Tsukamoto from $\bZ$-actions to actions of sofic groups. The relation between conditional metric mean dimension with relative mean dimension is discussed.

Consider a continuous pseudometric $\rho_X$ on a compact metrizable space $X$ and any $\varepsilon >0$. Recall that a subset $Z$ of $X$ is called {\it $(\varepsilon, \rho_X)$-separated} if $\rho_X(z_1, z_2) \geq \varepsilon$ for all distinct $z_1, z_2 \in Z$. Set $N_\varepsilon(X, \rho_X):=\max_Z|Z|$ for $Z$ ranging over all $(\varepsilon, \rho_X)$-separated subsets of $X$.

In the following, we adapt these notions to the context of factor maps as in \cite[Definition 3.4]{Lu17}. Let  $\pi \colon X \to Y$ be a  factor map between two dynamical systems. 

\begin{definition}
Fix $\varepsilon, \theta > 0$ and  $F \in \cF(\Gamma)$. Consider $\rho_Y$ as a continuous pseudometric on $Y$. For a map $\sigma \colon \Gamma \to \Sym(d)$ and a subset $\mathcal{K} \subseteq X^{[d]}$, we say a subset $\mathcal{E} \subseteq \mathcal{K}$ is {\it $(\varepsilon, \theta, \rho_{X, \infty}|\rho_{Y, \infty})$-separated} if for all distinct $\varphi, \varphi' \in \mathcal{E}$, one has
$$\rho_{Y, \infty}(\pi\circ \varphi, \pi \circ \varphi') \leq \theta \  {\rm and} \ \rho_{X, \infty}(\varphi, \varphi') \geq \varepsilon.$$  We call $\mathcal{E}$ is {\it $(\varepsilon, \rho_{X, \infty}|Y)$-separated} if $\pi\circ \varphi = \pi \circ \varphi'$ and $\rho_{X, \infty}(\varphi, \varphi') \geq \varepsilon$ for every distinct $\varphi, \varphi' \in \mathcal{E}$.

Denote by $N_\varepsilon (\mathcal{K}, \theta, \rho_{X, \infty}|\rho_{Y, \infty})$ the maximal cardinality of all $(\varepsilon, \theta, \rho_{X, \infty}|\rho_{Y, \infty})$-separated subsets of $\mathcal{K}$. Similarly, denote by $N_\varepsilon (\mathcal{K}, \rho_{X, \infty}|Y)$ the maximal cardinality of all $(\varepsilon, \rho_{X, \infty}|Y)$-separated subsets of $\mathcal{K}$. Equivalently, we have 
$$N_\varepsilon(\cK, \rho_{X, \infty}|Y)=\max_{\psi \in Y^d} N_\varepsilon(\cK\cap (\pi^d)^{-1}(\psi), \rho_{X, \infty}).$$

\end{definition}

Let $\Sigma=\{\sigma_i \colon \Gamma \to \Sym(d_i)\}_{i=1}^\infty$ be a sofic approximation of $\Gamma$.
\begin{definition} \cite[Definition 3.4]{Lu17}
Let $\pi \colon X \to Y$ be a factor map. Define
$$h_\Sigma(\rho_X|\rho_Y):=\sup_{\varepsilon > 0} \inf_{\theta > 0} \inf_{F \in \cF(\Gamma)} \inf_{\delta > 0} \limsup_{i\to \infty}\frac{1}{d_i}\log N_\varepsilon(\Map(\rho_X, F, \delta, \sigma_i), \theta, \rho_{X, \infty}|\rho_{Y, \infty}). $$
When $\varepsilon, F, \delta$ are fixed, denote the above limit supremum as $h_\Sigma^\varepsilon(\rho_X|\rho_Y, \theta, F, \delta)$. If  $\Map(\rho, F, \delta, \sigma_i)$ is empty for all sufficiently large $i$, the limit supremum is set to be $-\infty$. In a similar approach as in Definition \ref{top def of mdim}, we have the notations $h_\Sigma^\varepsilon(\rho_X|\rho_Y, \theta)$ and  $h_\Sigma^\varepsilon(\rho_X|\rho_Y)$.

Simultaneously, we set 
$$h_\Sigma(\rho_X|Y):=\sup_{\varepsilon > 0}  \inf_{F \in \cF(\Gamma)} \inf_{\delta > 0} \limsup_{i\to \infty}\frac{1}{d_i}\log N_\varepsilon(\Map(\rho_X, F, \delta, \sigma_i), \rho_{X, \infty}|Y). $$
Moreover, in the same approach as above, we introduce the notations $h_\Sigma^\varepsilon(\rho_X|Y, F, \delta)$, etc.

 When $\rho_X$ and $\rho_Y$ are dynamically generating, the {\it conditional sofic topological entropy $h_\Sigma(X|Y)$ of $\Gamma \curvearrowright X$ relative to $\Gamma \curvearrowright Y$} is defined as  the value $h_\Sigma(\rho_X|\rho_Y)$.
By \cite[Theorem 3.7]{Lu17}, $h_\Sigma(X|Y)$  is independent of the choice of dynamically generating continuous pseudometrics $\rho_X$ and $\rho_Y$. Notably,  when $Y$ is trivial as a singleton, $h_\Sigma(X|Y)$ recovers as the {\it sofic topological entropy $h_\Sigma(X)$ of $\Gamma \curvearrowright X$}. Again, following the same convention as above we have the notations $h_\Sigma^\varepsilon(\rho_X|Y)$, etc.
\end{definition}

A celebrated result of Luo is the equivalence between the conditional topological entropy $h_\Sigma(\rho_X|\rho_Y)$ and $h_\Sigma(\rho_X|Y)$ as follows.
\begin{proposition} \cite[Proposition 3.17]{Lu17} \label{metric celebrated}
    When $\rho_X$ and $\rho_Y$ are dynamical generating, for any $\varepsilon > 0$, the following holds
    $$h_\Sigma^\varepsilon(\rho_X|Y) \leq h_\Sigma^\varepsilon(\rho_X|\rho_Y) \leq h_\Sigma^{\varepsilon/2}(\rho_X|Y).$$    
    In particular, $h_\Sigma(\rho_X|\rho_Y)=h_\Sigma(\rho_X|Y)$.
\end{proposition}

\begin{definition}
For a factor map $\pi \colon X \to Y$, we define
    $$\mdim_{\Sigma, \rM} (\rho_X|\rho_Y) =\liminf_{\varepsilon \to 0} \frac{1}{\log(1/\varepsilon)} h_\Sigma^\varepsilon(\rho_X|\rho_Y).$$
and    $$\mdim_{\Sigma, \rM} (\rho_X|Y) =\liminf_{\varepsilon \to 0} \frac{1}{\log(1/\varepsilon)} h_\Sigma^\varepsilon(\rho_X|Y).$$
\end{definition}


From Proposition \ref{metric celebrated}, we immediately obtain:
\begin{theorem} \label{equivalent metric definitions}
     For any dynamically generating pseudometric $\rho_X$ and $\rho_Y$ on $X$ and $Y$ respectively, we have $\mdim_{\Sigma, \rM}(\rho_{X}|\rho_Y)=\mdim_{\Sigma, \rM}(\rho_X|Y)$. 
\end{theorem}

In light of this theorem, we introduce the following definition.
\begin{definition}
   For any continuous pseudometric $\rho_X$ on $X$, we define the {\it conditional sofic metric mean dimension of $\Gamma \curvearrowright X$ relative to $\Gamma \curvearrowright Y$ with respect to $\rho_X$} as 
   $$\mdim_{\Sigma, \rM}(X|Y, \rho_X):=\mdim_{\Sigma, \rM}(\rho_{X}|\rho_Y)$$
   for any dynamically generating pseudometric $\rho_Y$ on $Y$.
\end{definition}


\begin{remark}
   For every $1 \leq p < \infty$ we can also discuss the $p$-version of metric mean dimension (by substituting $\rho_{X, \infty}$ with $\rho_{X, p}$), akin to the concept presented in \cite[Chapter 3]{Lu17}. From the argument of \cite[Proposition 3.17]{Lu17}, one can also deduce a $p$-version of Theorem \ref{equivalent metric definitions}.
\end{remark}

To show that the conditional sofic metric mean dimension is indeed a generalization upon from amenable group actions, we need to first recall the definition of conditional metric mean dimension for amenable group actions introduced in \cite[Definition 4.1]{L22}.  
\begin{definition}
    Let $\Gamma$ be an amenable group and $\rho_X$ a dynamically generating continuous pseudometric on $X$. The {\it conditional metric mean dimension of $\Gamma \curvearrowright X$ relative to $\Gamma \curvearrowright Y$ with respect to $\rho_X$} is defined as
    $$\mdim_\rM(X|Y,\rho_X):=\liminf_{\varepsilon \to 0} \frac{1}{\log(1/\varepsilon)} \limsup_F \frac{\max_{y \in Y} \log N_\varepsilon(\pi^{-1}(y), \rho_{X, F})}{|F|}.$$
    When $\varepsilon$ is fixed, the above limit supremum is denoted by $h_\varepsilon(\rho_X|Y)$.
\end{definition}

\begin{theorem} \label{metric reduction}
Let $\pi \colon X \to Y$ be a factor map under the actions of a countable amenable group $\Gamma$. Fix a dynamically generating continuous pseudometric $\rho_X$ on $X$. 
Then
$\mdim_{\Sigma, \rM}(X|Y, \rho_X) = \mdim_\rM(X|Y, \rho_X)$.
\end{theorem}

\begin{proof}
 Fix a dynamically generating continuous pseudometric $\rho_Y$ on $Y$ and  $\varepsilon > 0$.   From the proof of \cite[Lemma 4.8]{Lu17}, we infer that  for every $\beta > 0$,
 $$h_\varepsilon(\rho_X|Y) \leq h^\varepsilon_\Sigma(\rho_X|\rho_Y) + 2\beta.$$
Consequently, we obtain $\mdim_\rM(\rho_X|Y) \leq \mdim_{\Sigma, \rM}(\rho_X|\rho_Y)$.
From the proof of \cite[Lemma 4.9]{Lu17}, we infer that $ h^\varepsilon_\Sigma(\rho_X|Y) \leq h_{\varepsilon/5}(\rho_X|Y)$ and hence 
$$\mdim_{\Sigma, \rM}(\rho_X|Y) \leq \mdim_\rM(\rho_X|Y).$$
By Theorem \ref{equivalent metric definitions}, the desired equality follows.
\end{proof}

Adapting the proof of \cite[Lemma 4.4]{Li13} into the conditional case, we obtain

\begin{proposition} \label{suppressing equality}
    Let $\rho_X$  be a dynamically generating pseudometric on $X$ and the compatible metric $\widetilde{\rho_X}$ induced as in (\ref{compatible induced}).  Then for any factor map $\pi \colon X \to Y$, we have
    $$\mdim_{\Sigma, \rM}(X|Y, \rho_X)=\mdim_{\Sigma, \rM}(X|Y, \widetilde{\rho_X}).$$
\end{proposition}

\begin{proof}
Fix a compatible metric $\rho_Y$ on $Y$. From  the proof of \cite[Lemma 4.4]{Li13},  for any $\varepsilon > 0$, we have $h_\Sigma^{4\varepsilon}(\widetilde{\rho_X}|Y)\leq h_\Sigma^{\varepsilon}(\rho_X|Y)\leq h_\Sigma^{\varepsilon/2}(\widetilde{\rho_X}|Y)$ and hence the conclusion follows.
\end{proof}


Building upon the results of \cite[Chapter 5]{Lu17} for $G$-extensions, we can extend the computation result in \cite[Proposition 4.3]{L22} to the sofic context. 
\begin{proposition}  \label{metric G-extension}
    Let $\pi \colon X \to Y$ be a $G$-extension map. Suppose that $\rho_X$ and  $\rho_G$ are compatible metrics on $X$ and $G$ respectively satisfying that 
    $$\rho_X(xg, xg')=\rho_G(g, g')$$
    for every $x \in X$ and $g, g' \in G$. Then $\mdim_{\Sigma, \rM}(\rho_X|Y)\leq \mdim_{\Sigma, \rM}(\rho_G)$
    and the equality holds if $\mdim_{\Sigma, \rM}(\rho_X)\geq 0$.
\end{proposition}

\begin{proof}
  Under our assumption on $\rho_X$ and $\rho_G$, for every $\varepsilon > 0$, the proof of \cite[Lemma 5.3]{Lu17} implies that $h_{\Sigma}^\varepsilon (\rho_X|Y)\leq h_{\Sigma}^\varepsilon(\rho_G)$. Moreover, if $\mdim_{\Sigma, \rM}(\rho_X)\geq 0$, the proof of \cite[Lemma 5.7]{Lu17} implies that $h_{\Sigma}^\varepsilon (\rho_X|Y)\geq h_{\Sigma}^\varepsilon(\rho_G)$. Hence our conclusion follows. 
\end{proof}

The following is a variant of Gromov's waist inequality \cite[Corollary 2.4]{ST23}.
\begin{lemma} \label{waist inequality}
  Let $n \geq m \geq 1$ be two integers. Then for any continuous map $f \colon [0, 1]^n \to \bR^m$ there exists $p \in \bR^m$ such that for every $0 < \varepsilon < 1/2$, one has
  $$N_{\varepsilon/2}(f^{-1}(p), ||\cdot||_\infty) \geq \frac{1}{8^n} \left(\frac{1}{\varepsilon}\right)^{n-m}$$
where $||\cdot||_\infty$ is understood as the metric induced from the $\ell^\infty$-norm on $[0, 1]^n$.
\end{lemma}

\begin{proof}
  For any closed subset $K \subseteq [0, 1]^n$, denote by $S_\varepsilon(K, ||\cdot||_\infty)$ the minimal number $|\cV|$ of a finite open cover $\cV$ of $K$ such that ${\rm diam} (V, ||\cdot||_\infty) < \varepsilon$ for every $V \in \cV$.  Then it is readily checked that $N_{\varepsilon/2}(K, ||\cdot||_\infty) \geq S_\varepsilon(K, ||\cdot||_\infty)$.  From \cite[Corollary 2.4]{ST23}, there exists $p \in \bR^m$ such that 
    $$ S_{\varepsilon}(f^{-1}(p), ||\cdot||_\infty) \geq \frac{1}{8^n} \left(\frac{1}{\varepsilon}\right)^{n-m}.$$
Then the conclusion follows.
\end{proof}

The following classical result of Hurewicz says that although an $n$-dimensional space $X$ may not embeds into $[0, 1]^{n+1}$ in general,  
there are rich of continuous maps from $X$ to $[0, 1]^{n+1}$ with controlled fibers \cite[Corollary 2.4]{T23} \cite{K68}. 
\begin{lemma} \label{Kur}
    Let $X$ be an $n$-dimensional compact metrizable space. Consider the space $C(X, [0,1]^{n+1})$ comprising all continuous maps from $X$ to $[0, 1]^{n +1}$ equipped with the topology induced from the $L^\infty$-norm. Then there exists a dense $G_\delta$ subset $\mathcal{A}$ of $C(X, [0,1]^{n+1})$ such that for any $f \in \mathcal{A}$, one has 
    $$\sup_{y \in [0, 1]^{n+1}} |f^{-1}(y)| \leq n+1.$$
\end{lemma}

Now we extend \cite[Theorem 1.5]{ST23}  and its improved version \cite[Theorem 1.2]{T23} into the sofic context.

\begin{theorem} \label{waist app2}
Let $a$ be a natural number and $\Gamma \curvearrowright Y$ any dynamical system. Equip $X=([0, 1]^a)^\Gamma$ with the left shift action of $\Gamma$. Define a dynamically generating pseudometric on $X$, given by $\rho_X(x, x'):=||x(e_\Gamma)-x'(e_\Gamma)||_\infty$  and let  $\widetilde{\rho_X}$ be the compatible metric induced as in Proposition \ref{suppressing equality}. Then for any continuous $\Gamma$-equivariant map  $\pi \colon  ([0, 1]^a)^\Gamma \to  Y$, we have
$$\mdimsm(X|Y, \widetilde{\rho_X}) \geq a -\mdims(Y).$$
In particular, if $a -\mdims(Y) > 0$, we have $h_\Sigma(X|Y)=+\infty$ .
\end{theorem}

\begin{proof}
Fix a compatible metric compatible metric $\rho_Y$ on $Y$.  Write $K=[0, 1]^a$. Fix $0 < \varepsilon < 1/2$ and $\theta > 0$. Then there exist $\delta > 0$ and $F \in \cF(\Gamma)$ such that 
\begin{equation} \label{Fdelta reduction}
    \mdim_\Sigma^\theta(\rho_Y, F, \delta) \leq \mdim_\Sigma^\theta(\rho_Y) +\theta.
\end{equation}

Since $\rho_X$ is dynamically generating, by Lemma \ref{Map into}, there exist $\delta' >0$ and $F' \in \cF(\Gamma)$ such that 
$$\pi^d \left(\Map(\rho_X, F', \delta', \sigma)\right) \subseteq \Map(\rho_Y, F, \delta, \sigma)$$
for any good enough map $\sigma \colon \Gamma \to \Sym(d)$.
Let  $\Psi$ be the continuous map from  $K^d$ to $X^d$  defined  by sending $x$ to $\varphi_x=((x(\sigma_{t^{-1}}(v)))_{t \in \Gamma})_{v \in [d]}$. Clearly for every $x, x' \in K^d$ we have
\begin{equation} \label{metric comparison}
||x-x'||_\infty \leq \rho_{X, \infty}(\Psi(x), \Psi(x')).
\end{equation}
When $\sigma$ satisfies $|\{v \in [d]: \sigma_{e_\Gamma}\circ \sigma_s(v)=\sigma_s(v), \ {\rm for \ every \ } s \in F'\}| > 1-\delta'$, one has
$\Psi(K^d) \subseteq \Map(\rho_X, F', \delta', \sigma)$. From Proposition \ref{suppressing equality}, we have
$\mdimsm(X|Y, \widetilde{\rho_X})=\mdimsm(\rho_X|\rho_Y)$. Thus it reduces to show 
    $$h_\Sigma^{\varepsilon/2}(\rho_X|\rho_Y, \theta, F', \delta')  \geq  (a-\mdim_\Sigma^\theta(\rho_Y) +\theta)\log(1/\varepsilon) -a \log 8.$$

Let $\Phi_0 \colon \Map(\rho_Y, F, \delta, \sigma) \to P$ be a $(\theta, \rho_{Y, \infty})$-embedding for some compact metrizable space $P$ with $\dim P =\Wdim_\theta(\Map(\rho_Y, F, \delta, \sigma), \rho_{Y, \infty})$. According to Lemma \ref{Kur}, there exists a continuous map $g \colon P \to [0, 1]^{\dim P +1}$ such that 
$$\sup_{p \in [0, 1]^{\dim P +1}} |g^{-1}(p)| \leq \dim P +1.$$
As $\sigma$ is good enough the map $\Pi =\Phi_0 \circ \pi^d \circ \Psi \colon K^d \to P$ is well defined.  It then induces a continuous map $\Phi \colon K^d \to  \bR^{\dim P +1}$ sending $x$ to $g\circ\Pi(x)$.

Applying Lemma \ref{waist inequality} to $\Phi$, we obtain $\xi \in P $ such that 
\begin{equation*} 
    N_{\varepsilon/2}(\Phi^{-1}(\xi), ||\cdot||_\infty) \geq \frac{1}{8^{ad}}\left(\frac{1}{\varepsilon}\right)^{ad-\dim P -1}.
\end{equation*}
Since $|g^{-1}(\xi)| \leq \dim P +1 $, we can find $\eta \in g^{-1}(\xi)$ such that 
\begin{equation} \label{waist point again}
    N_{\varepsilon/2}(\Pi^{-1}(\eta), ||\cdot||_\infty) \geq \frac{1}{8^{ad}(\dim P +1)}\left(\frac{1}{\varepsilon}\right)^{ad-\dim P -1}.
\end{equation}
Let $\mathcal{E} \subseteq \Pi^{-1}(\eta)$ be a maximal $(\varepsilon/2, ||\cdot||_\infty)$-separated subset. For any $x, x' \in \mathcal{E}$, we have $\Phi_0(\pi\circ \varphi_x)=\eta=\Phi_0(\pi\circ \varphi_{x'})$. By the choice of $\Phi_0$, it follows that 
$\rho_{Y, \infty}(\pi \circ \varphi_x, \pi \circ \varphi_{x'}) < \theta$. Thus $\Psi(\mathcal{E})$ is $(\varepsilon/2, \theta, \rho_{X, \infty}|\rho_{Y, \infty})$-separated. Using (\ref{metric comparison}) and (\ref{waist point again}), we obtain
$$N_{\varepsilon/2}(\Map(\rho_X, F', \delta', \sigma), \theta, \rho_{X, \infty}|\rho_{Y, \infty}) \geq |\Psi(\mathcal{E})| =|\mathcal{E}|\geq \frac{1}{8^{ad}(\dim P +1)}\left(\frac{1}{\varepsilon}\right)^{ad-\dim P -1}.$$
In view of (\ref{Fdelta reduction}), we obtain
\begin{align*}
    h_\Sigma^{\varepsilon/2}(\rho_X|\rho_Y, \theta, F', \delta') & \geq (a-\mdim_\Sigma^\theta(\rho_Y, F, \delta))\log(1/\varepsilon) -a \log 8 \\
    & \geq  (a-\mdim_\Sigma^\theta(\rho_Y )+\theta)\log(1/\varepsilon) -a \log 8.
\end{align*}

\end{proof}

\begin{theorem} \label{bridge inequality}
   Let $\pi \colon X \to Y$ be a factor map. For any compatible metric $\rho_X$ on $X$, we have $\mdims(\pi) \leq \mdim_{\Sigma, \rM}(X|Y, \rho_X)$. In particular, if $h_\Sigma(X|Y) < \infty$, we have $\mdims(\pi) \leq 0$.

\end{theorem}

\begin{proof}
   We adapt the proof of \cite[Theorem 6.1]{Li13} into our situation. Fix $\beta >0$. Without loss of generality, we may assume $D:=\mdim_{\Sigma, \rM}(\rho_X|Y) < \infty$. For any $\varepsilon > 0$, our goal is to show
   $\mdim_{\Sigma}^\varepsilon(\pi, F, \delta) \leq D +\beta$ for some $F \in \cF(\Gamma)$ and $\delta > 0$.

Choose a finite open cover $\cU$ of $X$ such that ${\rm diam}(U, \rho_X) < \varepsilon$ for every $U \in \cU$. Applying \cite[Lemma 6.2]{Li13} to $\cU$, there exists a Lipschitz function $f_U \colon X \to [0, 1]$ vanishing on $X\setminus U$ such that $\max_{U \in \cU} f_U(x)=1$ for all $x \in X$.  For any map $\sigma \colon \Gamma \to \Sym(d)$ it then induces a continuous map $\Phi \colon X^d \to [0, 1]^{\cU \times [d]}$ sending $\varphi$ to $(f_U(\varphi(v)))_{U \in \cU, v\in [d]}$. Drawing from the proof of \cite[Lemma 6.3]{Li13}, there exist $\delta > 0$ and  $F \in \cF(\Gamma)$ satisfying that,  for good enough $\sigma$, for each $\psi \in Y^d$, there exists $\xi_\psi \in (0, 1)^{\cU \times [d]}$  such that 
$$\xi_\psi |_S \notin \Phi(\Map(\rho_X, F, \delta, \sigma)\cap  
(\pi^{d})^{-1}(\psi))|_S$$
for every $S \subseteq \cU \times [d]$ with $|S| \geq (D+\beta)d$. 

For each $\psi \in Y^d$ write $\cK_\psi=\Map(\rho_X, F, \delta, \sigma)\cap (\pi^{d})^{-1}(\psi)$ and $\mathcal{Z}_\psi =\Phi(\cK_\psi)$.
 Put $m=\lfloor(D+\beta)d\rfloor$ and $W=\cU\times [d]$. From the argument of \cite[Lemma 6.5, Lemma 6.4]{Li13}, we can use $\xi_\psi$ to construct a continuous map $\Psi_\psi \colon \mathcal{Z}_\psi \to [0, 1]^W$ such that $\dim \Psi_\psi(\mathcal{Z}_\psi) \leq m$, and further, the map $\Psi_\psi\circ \Phi \colon \cK_\psi \to \Psi_\psi(\mathcal{Z}_\psi)$ makes an $(\varepsilon, \rho_{X, \infty})$-embedding due to the choice of $\cU$. Consequently, for any $\psi \in Y^d$, we obtain
$$\Wdim_\varepsilon(\cK_\psi, \rho_{X, \infty}) \leq \dim \Psi_\psi(\mathcal{Z}_\psi) \leq (D+\beta)d$$
and hence $\mdim_\Sigma^\varepsilon(\pi, F, \delta) \leq D+\beta$.

\begin{example}
    Let $\Gamma$ be the free group generated by $a$ and $b$ and $\bT=\bR/\bZ$.  The Ornstein-Weiss factor map $\pi \colon \bT^\Gamma \to (\bT^\Gamma)^2$ is defined by sending $x=(x_s)_{s \in \Gamma}$ to $(x_s-x_{sa}, x_s-x_{sb})_{s \in \Gamma}$.   The map $\pi$ is a surjective continuous group homomorphism with the kernel isomorphic to $\bT$ and hence is a $\bT$-extension. Set $X=\bT^\Gamma$ and $Y=(\bT^\Gamma)^2$. Let $\rho_\bT$ be the compatible metric  on $\bT$ defined by 
    $$\rho_\bT(x+\bZ, y+\bZ):=\min_{k \in \bZ}|x-y+k|.$$ 
    It induces a dynamically generating pseudometric $\rho_X$ on $\bT^\Gamma$ defined similarly as in Proposition \ref{waist app2}.
    By Theorem \ref{bridge inequality} and Proposition \ref{metric G-extension}, we have
    $$\mdims(\pi)\leq \mdim_{\Sigma, \rM}(\rho_X|Y)=\mdim_{\Sigma, \rM}(\bT, \rho_\bT)=0.$$
Note that the identity of $X$ is a fixed point of $\Gamma \curvearrowright X$. It follows that  $\mdims(\pi) \geq 0$ and hence  $\mdims(\pi) = 0$.  

\end{example}


\end{proof}

\section{Question}

A remaining question on conditional sofic mean dimension is as follows:
\begin{question}
Let $\pi \colon X \to Y$ be a factor map under the actions of a sofic group $\Gamma$. Suppose that $\rho_X$ and $\rho_Y$ are compatible metrics on $X$ and $Y$ respectively.  Do we have the equivalence that $\mdims(\rho_X|\rho_Y) = \mdims(\rho_X|Y)$?\\
\end{question}


\end{document}